\documentclass[11pt, a4paper]{amsart}
\usepackage{amsfonts,amssymb,amsmath, amsthm}
\usepackage{mathrsfs}
\usepackage{lmodern}
\usepackage{qsymbols}
\usepackage{latexsym}
\usepackage[noadjust]{cite}
\usepackage{pdfsync}
\usepackage{bm}
\usepackage{enumitem}

\newtheorem{thm}{Theorem}[section]
\newtheorem{lem}[thm]{Lemma}

\newtheorem{cor}[thm]{Corollary}

\theoremstyle{definition}

\newtheorem{rem}[thm]{Remark}

\numberwithin{equation}{section}
\usepackage[plainpages=false,pdfpagelabels,backref=page,citecolor=red]{hyperref}
\usepackage{xcolor}
\hypersetup{
colorlinks,
linkcolor={cyan!90!black},
citecolor={magenta},
urlcolor={green!40!black}
} 
\newcommand{\R}{\mathbb{R}}
\newcommand{\IC}{\mathbb{C}}

\newcommand{\cH}{\mathcal{H}}

\newcommand{\loc}{\operatorname{loc}}

\renewcommand{\L}{\operatorname{L}} 
\renewcommand{\H}{\operatorname{H}}
\renewcommand{\H}{\operatorname{H}} 
\newcommand{\E}{\mathsf{E}} 

\newcommand{\ree}{{\mathbb{R}^{n+1}}}
\renewcommand{\P}{\mathcal{P}} 

\newcommand{\gradx}{\nabla_x}

\renewcommand{\div}{\operatorname{div}}

\newcommand{\dhalf}{D_t^{1/2}} 
\newcommand{\HT}{H_t} 

\renewcommand{\d}{\, \mathrm{d}} 
\renewcommand\Re{\operatorname{Re}}
\renewcommand\Im{\operatorname{Im}}

\newcommand{\cl}[1]{\overline{#1}} 
\DeclareMathOperator{\dist}{dist}
\DeclareMathOperator{\dom}{\mathsf{D}} 

\newcommand{\sgn}{\operatorname{sgn}}

\newcommand{\BMO}{\textup{BMO}}

\def\Xint#1{\mathchoice
{\XXint\displaystyle\textstyle{#1}}%
{\XXint\textstyle\scriptstyle{#1}}%
{\XXint\scriptstyle\scriptscriptstyle{#1}}%
{\XXint\scriptscriptstyle%
\scriptscriptstyle{#1}}%
\!\int}
\def\XXint#1#2#3{{\setbox0=\hbox{$#1{#2#3}{%
\int}$ }
\vcenter{\hbox{$#2#3$ }}\kern-.6\wd0}}
\def\barint{\,\Xint -} 
\def\bariint{\barint_{} \kern-.4em \barint}
\def\bariiint{\bariint_{} \kern-.4em \barint}
\renewcommand{\iint}{\int_{}\kern-.34em \int} 
\renewcommand{\iiint}{\iint_{}\kern-.34em \int} 

\title[Generalizations of the parabolic Kato square root estimate]{The Kato square root problem for parabolic operators with an anti-symmetric part in  BMO}

\author{Alireza Ataei}
\email{alireza.ataei@math.uu.se}
\address{Department of Mathematics, Uppsala University, S-751 06 Uppsala,
Sweden}

\author{Kaj Nystr\"{o}m}\thanks{K.N. was partially supported by
grant 2022-03106 from the Swedish research council (VR)}
\email{kaj.nystrom@math.uu.se}
\address{Department of Mathematics, Uppsala University, S-751 06 Uppsala,
Sweden}

\thanks{}

\subjclass[2010]{Primary: 35K10, 35K20; Secondary: 26A33, 42B25}
\keywords{Kato square root problem, square function, parabolic operator, second order parabolic operator, BMO, anti-symmetric.}
\date{\today}

\begin{document}
\begin{abstract}
We solve the Kato square root problem for parabolic operators whose coefficients can be written as the sum of a complex part, which is coercive, and a real anti-symmetric part, which is in $\BMO$. In particular, we allow for certain unbounded coefficients.
\end{abstract}

\maketitle


\section{Introduction and  the main result}

 In the variables $(x,t) \in \R^n \times \R=:\ree$, we consider parabolic operators of the form
\begin{eqnarray}\label{eq1deg+}
\cH:=\partial_t -\div_{x} (A(x,t)\nabla_{x} ),
 \end{eqnarray}
where  the coefficient matrix $A = A(x,t)$ is measurable with complex entries. Furthermore,
\begin{eqnarray}\label{eq1deg+h}
 A(x,t)=S(x,t)+D(x,t),
 \end{eqnarray}
 where $S = S(x,t)$ is  complex and  coercive, while $D = D(x,t)$ is real, measurable, anti-symmetric, having entries in a space of functions of bounded mean oscillation functions. In particular, $A$ may be unbounded. We refer to \eqref{ellip} and \eqref{bmocoeff} below for the precise formulations of our assumptions on $A=S+D$.

Parabolic operators of the form \eqref{eq1deg+} and \eqref{eq1deg+h}, with \( A \) real, have recently been investigated in several works, most notably in \cite{seregin2012divergence, qian2019parabolic}. In \cite{seregin2012divergence}, motivated by questions concerning the behavior of solutions to elliptic and parabolic equations with low-regularity drift terms, the authors studied equations of the form
\[
\partial_t u + {\bf c} \cdot \nabla_x u - \Delta u = 0,
\]
as well as their stationary counterparts, where \( {\bf c} \) is a divergence-free vector field in \( \mathbb{R}^n \). They demonstrated that the divergence-free condition on \( {\bf c} \) allows for relaxed regularity assumptions on \( {\bf c} \), under which the Harnack inequality and other regularity results for solutions can still be obtained. Notably, the interior regularity theory of De Giorgi, Nash, and Moser extends to these parabolic equations if \( {\bf c} \in \L^\infty(\mathbb{R}, (\BMO(\mathbb{R}^n, \mathrm{d}x))^{-1}) \).

For more general parabolic equations in divergence form, this condition corresponds to considering operators as in \eqref{eq1deg+} and \eqref{eq1deg+h}, with \( A \) real, under the assumption that \( S \) satisfies \eqref{ellip} and \( D \) is anti-symmetric and satisfies \eqref{bmocoeff}. The space of bounded mean oscillation functions, \( \BMO \), plays a central role here. First, \( \BMO \) exhibits the correct scaling properties that arise naturally in the iterative arguments of De Giorgi-Nash-Moser. Second, the \( \BMO \) condition on the anti-symmetric part of the matrix ensures the proper definition of weak solutions. This follows essentially from the div-curl lemma in the theory of compensated compactness \cite{coifman1993compensated}, as outlined in \cite{seregin2012divergence}.

In \cite{qian2019parabolic}, the authors proved Aronson-type bounds for the fundamental solutions of the operator
\[
\partial_t - \div_x((S + D) \nabla_x).
\]

The studies in \cite{seregin2012divergence, qian2019parabolic} are motivated by extending the De Giorgi-Nash-Moser framework to operators with possibly unbounded coefficients. These works are also inspired by the relevance of such operators to the Navier-Stokes system
\begin{align}\label{eq:ns01}
    (\partial_t + {\bf u} \cdot \nabla_x - \nu \Delta_x) {\bf u} &= \nabla_x p, \notag \\
    \nabla_x \cdot {\bf u} &= 0,
\end{align}
where ${\bf u} = (u_1, u_2, u_3)$ represents the velocity vector field of the fluid flow, and \( p(x, t) \) denotes the pressure at location \( x \) and time \( t \). The vorticity \( {\bf \omega} = (\omega_1, \omega_2, \omega_3) \) is related to the velocity \( {\bf u} \) via \( {\bf \omega} = \nabla_x \times {\bf u} \). By formally differentiating the Navier-Stokes system, the vorticity \( {\bf \omega} \) satisfies the equation
\begin{equation}\label{eq:vor1}
    (\partial_t + {\bf u} \cdot \nabla_x - \nu \Delta_x) {\bf \omega} = {\bf \omega} \cdot \nabla_x {\bf u}.
\end{equation}

The key feature of \eqref{eq:vor1} is that \( {\bf u}(x, t) \) is a time-dependent vector field with limited regularity, which is, however, solenoidal—meaning that for every \( t \), \( \nabla_x \cdot {\bf u}(\cdot, t) = 0 \) in the sense of distributions. Consequently, both \( \partial_t + {\bf u} \cdot \nabla_x - \nu \Delta_x \) and its adjoint \( -\partial_t - {\bf u} \cdot \nabla_x - \nu \Delta_x \) can be viewed as diffusion generators.

As noted in \cite{qian2019parabolic}, in three dimensions, the divergence-free property of \( {\bf u} = (u_1, u_2, u_3) \) implies
\[
{\bf u} \cdot \nabla_x = \sum_{i,j=1}^3 \partial_{x_i} (D_{i,j} \partial_{x_j}),
\]
where \( \{D_{i,j}\} \) is an anti-symmetric matrix. Thus, the operator \( \nu \Delta_x - {\bf u} \cdot \nabla_x \) can be expressed in divergence form as
\[
\sum_{i,j} \partial_{x_i} (\nu \delta^{ij} + D_{i,j}) \partial_{x_j} \equiv \div_x \big((S(x, t) + D(x, t)) \nabla_x \big),
\]
where the symmetric part \( S \) is uniformly elliptic, and the anti-symmetric part \( D \) determines the divergence-free drift vector field \( {\bf u} \).

\subsection{Statement of our main result} The purpose of this paper is to solve the Kato square root problem for $\cH$, that is, to prove Theorem \ref{thm:Kato} stated below. The space $\E(\ree)$ is defined as the space of square integrable functions $u$ whose spatial gradient $\nabla_x u$ and half order
derivative $\dhalf u$ are square integrable. For rigorous definitions we refer to Section \ref{sec1}.

\begin{thm}\label{thm:Kato} Assume that  $A=S+D$ satisfies \eqref{ellip} and \eqref{bmocoeff}.  Then, the part of $\cH$ in $\L^2(\ree)$, with maximal domain $\dom(\cH)=\{u \in \E(\ree): \cH u \in \L^2(\ree)\}$, is {maximal accretive}.  Furthermore, there is a well-defined square root of $\cH$ whose domain satisfies $\dom(\sqrt{\cH}) = \mathsf{E}(\ree)$, and
\begin{align*}
\|\sqrt {\cH}\, u\|_{2} \sim \|\nabla_x u\|_{2}+ \| \dhalf u\|_{2}  \qquad (u \in \mathsf{E}(\ree)),
\end{align*}
holds with  implicit constants that only depend on the dimension, the boundedness and coercivity  parameters of $S$, and the $\BMO$ constant of $D$. The same conclusions are true with $\cH$ replaced by $\cH^*$.
\end{thm}

Kato square root problems, along with their associated estimates, play a crucial role in operator theory and the study of boundary value problems for elliptic and parabolic equations and systems. These estimates have garnered significant attention since the resolution of the Kato square root conjecture for elliptic operators in \cite{AHLMcT} and the elliptic counterpart of Theorem \ref{thm:Kato} was proved in \cite{EH}.

Recently, notable progress has been made in the study of the parabolic Kato problem for operators of the form \eqref{eq1deg+} and \eqref{eq1deg+h} under the assumption \( D \equiv 0 \). Specifically, in \cite{AEN}, the Kato (square root) estimate was established, and in \cite{AAEN}, the case of a weighted operator of the form
\[
\partial_t - w(x)^{-1}\div_{x}(A(x, t)\nabla_{x})
\]
was addressed, assuming \( w = w(x) \) is a real-valued function belonging to the Muckenhoupt class \( A_2(\mathbb{R}^n, \mathrm{d}x) \).

The field of parabolic Kato square root estimates originated in \cite{N1}, where a notion of reinforced weak solutions for second-order parabolic equations, induced by a coercive sesquilinear form, was introduced, and the proof of the original Kato conjecture for elliptic operators \cite{AHLMcT} was adapted to the parabolic setting. For further details on parabolic Kato square root problems, we refer to \cite{AEN, AEN1, N1, CNS, N2, AAEN}.

Our contribution is the resolution of the Kato square root problem for parabolic operators of the form \eqref{eq1deg+} and \eqref{eq1deg+h}, under the assumptions \eqref{ellip} and \eqref{bmocoeff}. This framework allows for certain unbounded coefficients through \( D \). While this result is of independent interest, it is motivated by the long-term goal of establishing solvability and regularity for Dirichlet and Neumann problems in the upper half-space with data in \( \L^p \), for operators modeled on those in \eqref{eq1deg+} and \eqref{eq1deg+h}. Specifically, we aim to generalize the results of \cite{AEN1} to operators of this type, incorporating one additional spatial coordinate and considering \( S \) and \( D \) as real-valued.

In \cite{AEN1}, it was shown that for parabolic operators in the upper half-space, assuming \( D \equiv 0 \) and \( A \) is real, bounded, measurable, independent of the transversal variable, uniformly elliptic, but not necessarily symmetric, the associated parabolic measure is absolutely continuous with respect to the surface measure, as characterized by the Muckenhoupt class \( A_\infty \). Theorem \ref{thm:Kato}, along with the associated square function estimates, represents a necessary step towards generalizing these results to our setting. This is because the proof of the crucial Carleson measure estimate in \cite{AEN1} relies heavily on square function estimates, the Kato estimate, and the non-tangential estimates for parabolic operators with time-dependent coefficients established in \cite{AEN}.

In particular, to extend the framework of \cite{AEN1} to operators as in \eqref{eq1deg+} and \eqref{eq1deg+h}, it appears essential to first resolve the Kato square root problem for such operators. This necessity has driven our investigation, although we emphasize that the Kato square root problem is also of significant independent interest.

For contributions to the boundary behavior of solutions of elliptic operators in divergence form with a BMO anti-symmetric part,
the Dirichlet problem for such operators, and \(\L^p\) theory for the square roots and square functions of these operators,
see \cite{LP19}, \cite{HLMP22a}, and \cite{HLMP22b}. Notably, in \cite{HLMP22a}, the crucial estimate is presented in Lemma 2.4,
and its proof closely follows the argument in \cite{AEN1}, with minimal acknowledgment. This serves as an example where progress
on parabolic equations has significantly contributed to advancements on elliptic problems.

\subsection{The Proof of Theorem \ref{thm:Kato}} To prove Theorem \ref{thm:Kato}, we follow the approach outlined in \cite{AAEN} and \cite{N1}, both of which build on ideas developed in the proof of the original elliptic Kato problem in \cite{AHLMcT}. In \cite{AAEN}, revisiting \cite{N1}, a `second-order' method was introduced for parabolic operators with time-dependent measurable coefficients, including weighted operators with \( D \equiv 0 \). This approach significantly simplified the proof of \cite[Thm.~2.6]{AEN} by arranging the main quadratic/square function estimate in a way that nearly separates the time and space variables. In particular, at the level of off-diagonal bounds, only operators involving spatial differentiation require estimates. These estimates are directly deduced from the equation, respecting parabolic scaling, and avoid the more complex off-diagonal decay and Poincaré inequalities for non-local derivatives, which were key novelties in \cite{AEN}.

Operationally, the argument in \cite{AAEN} reorganized terms from \cite{N1} to ensure that fine harmonic analysis estimates apply solely to the spatial components, while \( t \)-derivatives appear in blocks manageable with elementary resolvent estimates based on the maximal accretivity and hidden coercivity of \( \mathcal{H} \).

A pivotal observation in these arguments is that when the time variable spans the full real line, parabolic operators exhibit `hidden coercivity,’ revealed through the Hilbert transform \( \HT \) in the \( t \)-variable. Decomposing \( \partial_t = \dhalf \HT \dhalf \), the sesquilinear form associated with \eqref{eq1deg+} in \( \L^2 \) is
\begin{align}
    \label{hidden_coercivity_intro}
    \iint_{\mathbb{R}^{n+1}} A \nabla_x u \cdot \overline{\nabla_x v} + \HT \dhalf u \cdot \overline{\dhalf v} \, \mathrm{d}x \mathrm{d}t, \quad (u, v \in \mathsf{E}(\mathbb{R}^{n+1})).
\end{align}
Lower bounds incorporating both time and space derivatives are then derived by taking \( v = (1+\delta \HT)u \) with \( \delta > 0 \) small. This technique, originating in Kaplan~\cite{Kaplan}, was rediscovered in \cite{N1} for the parabolic Kato problem. The maximal accretivity of \( \mathcal{H} \) follows directly from this observation.

On a technical level, a key contribution of this paper, compared to \cite{N1, AEN, AAEN}, is the derivation of off-diagonal estimates for operators with time-dependent coefficients, where the anti-symmetric part \( D \) lies in \( \BMO \). These estimates are established by separating the analysis into spatial and temporal components, as detailed in Lemma \ref{l0} and Lemma \ref{lenewpiece}, combining methods from \cite{AAEN}, \cite{N1}, and \cite{EH}. While one might consider using Lemma 1 in \cite{EH} by selecting a solution \( \mathcal{H} u = f \) for \( f \in \L^2(\mathbb{R}^{n+1}) \) within a compactly supported space, this is not feasible. The sesquilinear form defining \( \mathcal{H} \), see \eqref{hidden coercivity}, involves the Hilbert transform \( \HT \), which does not preserve compactly supported functions in \( \mathbb{R}^{n+1} \). Consequently, the idea of hidden coercivity and the Lax-Milgram theorem cannot directly apply. Instead, we overcome this by carefully separating the spatial and temporal components in our argument.

An additional key step in proving the Kato estimate involves deriving \( T(b) \)-type estimates and the crucial Carleson measure estimate for \( \mathcal{U}_\lambda A \), where \( \mathcal{U}_\lambda := \lambda (1 + \lambda^2 \mathcal{H})^{-1} \nabla_x \). While \( \mathcal{U}_\lambda A \) satisfies uniform \( \L^2 \) and off-diagonal estimates, the challenge lies in the fact that \( A = S + D \), with \( D \in \L^\infty(\mathbb{R}, \BMO(\mathbb{R}^n, \mathrm{d}x)) \). This complicates the definition of \( \mathcal{U}_\lambda A \) and the approximation of \( (\mathcal{U}_\lambda A) \mathcal{A}_\lambda \), where \( \mathcal{A}_\lambda \) is a self-adjoint averaging operator over parabolic cubes of size \( \lambda \).

To address this, we modify \( D \) in \( A = S + D \) to make \( D \) and \( A \) locally bounded in \( \L^2 \) (see Subsection \ref{keysec}). We then use the John-Nirenberg inequality to approximate \( \|(\mathcal{U}_\lambda A) \mathcal{A}_\lambda\|_2 \), resolving the issue and enabling our analysis.

\subsection{Organization of the Paper}
The remainder of the paper is dedicated to the proof of Theorem \ref{thm:Kato}. Section \ref{sec1} is partially preliminary in nature, introducing the functional setting of the paper. Section \ref{max} focuses on maximal functions and Littlewood-Paley theory. In Section \ref{form}, the underlying sesquilinear form is introduced, and maximal accretivity is established. Section \ref{off} is devoted to the proof of off-diagonal estimates, inspired by the corresponding arguments in \cite{EH}. Additional estimates based on these off-diagonal estimates and related to the principal part approximation are derived in Section \ref{additional}. Finally, the proof of Theorem \ref{thm:Kato} is presented in Section \ref{sec2}.

\section{Preliminaries}\label{sec1}

\subsection{Notation}
Given $(x,t)\in\mathbb R^{n}\times\mathbb R=\ree$, we let $\|(x,t)\|:=\max \{|x|, |t|^{1/2}\}$. We call $\|(x,t)\|$
the parabolic norm of $(x,t)$. For a cube $Q= Q_r(x) := (x-r/2,x+r/2]^n \subset \R^n$ with sidelength $r$ and center $x \in \R^n$, and an interval $I=I_r(t):= (t-r^2/2,t+r^2/2]$, we call $\Delta= \Delta_r(x,t):=Q\times I\subset\mathbb R^{n+1}$ a parabolic cube of size $\ell(\Delta):= r$. For every $c>0$, the dilation $c\Delta := cQ \times c^2 I$ is defined by the convention that $cQ := (x-cr/2,x+cr/2]^n $ and $c^2 I = (t-(c r)^2/2,t+(c r)^2/2]$ denote dilates of cubes and intervals, respectively, keeping the center fixed and dilating {their} radius by $c$ and $c^2$, {respectively}.

We will use the notation
\begin{align*}
\barint_{E} g :=\frac {1}{|E|}\int_{E} g(x)\, \d x,\  \bariint_{F} f &:=\frac {1}{|F|}\iint_{F} f(x,t)\, \d x \d t,
\end{align*}
  where $E \subset \mathbb{R}^{n}, F \subset \R^{n+1}$ are bounded Lebesgue measurable subsets and $g: \R^n \to \mathbb{C}^k$, $f: \R^{n+1} \to \mathbb{C}^k$ are locally integrable functions. Given a locally integrable function $f: \ree \to \mathbb{C}^k$, and a bounded Lebesgue measurable $E \subset \R^n$, we define
  \begin{align*}
     \bigl( \barint_{E} f \bigr)\,(t) &:=\frac {1}{|E|}\int_{E} f(x,t)\, \d x,
  \end{align*}
  for every $t \in \R.$ Note that $\barint_{E} f$ is well-defined a.e. on $\R$ and locally integrable by Fubini's theorem. By abuse of notation, we write
  $$\bigl( \barint_{E} f \bigr)\,(x,t) =  \bigl( \barint_{E} f \bigr)\,(t),$$ $x \in \R^n, t \in \R$, when we want to emphasize that $\barint_{E} f$ is considered as a function on $\ree.$

  Throughout the paper we let $\Omega \subset \mathbb{R}^{n+1}$ denote an open set. Given $p \geq 1$, the space  $\L^p(\Omega;\mathbb C^k)$ consists of measurable functions $u: \Omega \to \mathbb{C}^k$ such that
$$\|\cdot\|_{\L^p(\Omega)}=:=\biggl (\iint_{\Omega}|\cdot|^p\, \d x \d t\biggr )^{1/p}<\infty.$$ We let $\L^p:=\L^p(\ree):=\L^p(\ree;\mathbb C)$ and $ \|\cdot\|_{p} :=\|\cdot\|_{\L^p(\ree)}.$ We let $\|\cdot \|_{p \to p}$ denote the norm of an operator from $\L^p(\ree)$ to $\L^p(\ree)$. We define the inner product
$$
\langle f,g\rangle := \iint_{\ree} f \, \cl{g}\, \d x \d t,
$$
for $f ,g \in \L^2(\ree)$.

The spaces $W^{1,2}(\ree), \H^1(\R^n)$ consist of all functions $f\in\L^2(\ree)$, $g\in\L^2(\R^n)$, such that $\|\nabla_x f\|_2 + \|\partial_t f\|_2 < \infty$ and  $\|\nabla_x g\|_{\L^2(\R^n)} < \infty$, respectively. The norms for these spaces are
\begin{align*}
   \|f\|_{W^{1,2}(\ree)} &:=  \bigl(\|f\|^2_2 + \| \nabla_x f\|^2_2 + \| \partial_t f\|_2^2  \bigr)^{1/2},\\
    \|g\|_{\H^1(\R^n)} &:=  \bigl(\|g\|^2_{\L^2(\R^n)} + \| \nabla_x g\|^2_{\L^2(\R^n)} \bigr)^{1/2}.
\end{align*}

\subsection{Convention concerning constants}\label{coo} Concerning the coefficients of our operator, throughout the paper we will assume \eqref{ellip} and \eqref{bmocoeff} stated below. We will refer to $n$ and the constants $c_1,c_2,c_3$, appearing in \eqref{ellip} and \eqref{bmocoeff}, as structural constants. For all constants $A,B\in \mathbb R_+$, the notation $A\lesssim B$  means, unless otherwise stated, that $A/B$ is bounded from above by a positive constant depending at most on the structural constants.  $A\gtrsim B$ should be interpreted similarly. We write $A\sim B$  if $A\lesssim B$ and  $B\lesssim A$. $A\lesssim_\eta B$  means, unless otherwise stated, that $A/B$ is bounded from above by a positive constant depending at most on the structural constants and the constant $\eta$.

\subsection{The energy space}  We define the half-order $t$-derivative $\dhalf$ via the Fourier symbol $ |\tau|^{1/2}$ and we let
\begin{equation}\label{pargradient}
\mbox{${\mathbb D}:= (\nabla_x,\dhalf)$.\footnote{Note that in \cite{N1}, the notation ${\mathbb D}$ has a different meaning than in this paper.}}
\end{equation} The Hilbert transform in $t$, $\HT$, is defined through the Fourier symbol $i\mbox{sgn}(\tau)$.  The inhomogeneous energy space $\E=\E(\ree) $ is equipped with the  Hilbertian norm
\begin{align*}
\|f\|_{{\E}(\ree)}^2 := \|\gradx f\|^2_{2}+\|\dhalf f\|^2_{2}+\|f\|^2_{2}= \langle f,f\rangle_{\E(\ree)},
\end{align*}
where $\langle \cdot,\cdot\rangle_{\E(\ree)}$ is the naturally defined  inner product on $\E(\ree)$. I.e., $\E=\E(\ree)$ is the set of all $f\in\L^2(\ree)$ such that $\|f\|_{{\E}(\ree)}<\infty$. It can be proved that $C_0^{\infty}(\ree)$ is dense in $\E(\ree)$ in the norm $\|\,\|_{\E(\ree)},$ see Lemma 3.3 in \cite{AAEN}.
Let $\Omega=U\times\mathbb R \subset \ree$ where $U \subset \R^n$ is an open and bounded set. We define the subspace $\E(\Omega) \subset \E(\ree)$ to be the set of all functions $f \in \E(\ree)$ such that $f = 0$ a.e. on $(\mathbb R^n\setminus U)\times\mathbb R$. The induced norm and inner product on $\E(\Omega)$ are denoted by $\|\, \|_{\E(\Omega)}, \langle \, , \, \rangle_{\E(\Omega)}$, respectively. We will need the following lemma.

\begin{lem}
\label{space}
    Let $\Omega=U\times\mathbb R \subset \ree$ where $U \subset \R^n$ is an open and bounded set. Then the following statements are true.
     \begin{enumerate}
         \item The space $C^{\infty}_0(\Omega)$ is dense in $\E(\Omega)$.
         \item Multiplication by $C^1(\ree)$-functions is bounded on $\E(\Omega).$
         \item The Hilbert transform $\HT$ preserves $\E(\Omega)$.
     \end{enumerate}
\end{lem}
\begin{proof} {The argument is similar to \cite[Thm. 4.2]{Evans}.} Consider $f \in \E(\Omega)$ and let $\eta \in C^{\infty}_0(\ree)$ be a non-negative and radially symmetric approximation of the identity. Let $\eta_{\epsilon}(x,t) :=  {\epsilon^{-(n+1)}} \eta({x}/{\epsilon}, {t}/{\epsilon^2})$ for all $\epsilon >0, (x,t) \in \ree$. For $k\geq 1$, we let
\begin{align*}
 U_k := \bigl\{x\in  U:\ \frac{1}{k+1} \leq \dist(x,\partial U) < \frac{1}{k-1}\bigr\},
\end{align*}
where $\dist(x,\partial U)$ is the standard Euclidean distance {of} $x\in U$ {from} $\partial U$. We also introduce, for $m\geq 1$,
\begin{align*}
 I_m := \bigl\{t\in \mathbb R:\  m-1\leq |t| < m+1\bigr\}.
\end{align*}
We let $V_{k,m}:=U_k\times I_m$, and we introduce non-negative functions $\{\zeta_{k,m}\}$ such that $\zeta_{k,m} \in C_0^{\infty}(V_{k,m})$, and such that $$\sum_{k,m=1}^{\infty} \zeta_{k,m} =1 \quad \textup{ on } U\times \mathbb R.$$
{In fact}, $\{\zeta_{k,m}\}{_{k,m=1}^{\infty}}$ defines a partition of unity of $U\times \mathbb R$. Using  \cite[Lem. 3.3 (ii)]{AAEN}, we have $\zeta_{k,m} f \in \E(\ree)$. We claim that, that there exists, for each $k,m \geq 1$, $\epsilon_{k,m} >0$ such that $\eta_{\epsilon_{k,m}} \ast (\zeta_{k,m} f) \in C^{\infty}_0(V_{k,m})$ and
\begin{align*}
   \| \eta_{\epsilon_{k,m}} \ast (\zeta_{k,m} f) - \zeta_{k,m} f \|_{\E(\ree)} < \frac{\epsilon}{2^{k+m}}.
\end{align*}
The first part of the claim is immediate from the definition. For the second part, by \cite[Lem. 3.3 (i)]{AAEN}, we can assume, without loss of generality, that $f \in C^{\infty}_0(\ree).$ Hence, by standard approximation results in the Sobolev space $W^{1,2}(\ree)$ we see that  there exists $\epsilon_{k,m}>0$ small, such that $$\|\eta_{\epsilon_{k,m}} \ast (\zeta_{k,m} f) - \zeta_{k,m} f\|_{W^{1,2}(\ree)} \leq \frac{\epsilon}{2^{k+m+1}}.$$ Hence, by an application of the Fourier transform and Plancherel's theorem, we obtain
$$
\|\dhalf (\eta_{\epsilon_{k,m}} \ast (\zeta_{k,m} f) - \zeta_{k,m} f)\|^2_2 \leq  \|\partial_t (\eta_{\epsilon_{k,m}} \ast (\zeta_{k,m} f) - \zeta_{k,m} f)\|_2 \| (\eta_{\epsilon_{k,m}} \ast (\zeta_{k,m} f) - \zeta_{k,m} f)\|_2 \leq \frac{\epsilon^2}{4^{k+m+1}},
$$
which concludes the proof of the claim. Note that $$f = \sum_{k,m=1}^{\infty} \zeta_{k,m} f,$$ since $f =0$ a.e. on $\ree \setminus \Omega.$ Hence, for the sequence of $\epsilon_{k,m}$ given by the claim,
\begin{align*}
   \bigl \|\sum_{k,m=1}^{\infty} \eta_{\epsilon_{k,m}} \ast (\zeta_{k,m} f) - f \bigr\|_{\E(\ree)} \leq \sum_{k,m=1}^{\infty} \| \eta_{\epsilon_{k,m}} \ast (\zeta_{k,m} f) - \zeta_{k,m} f\|_{\E(\ree)} \leq \epsilon.
\end{align*}
This  concludes the proof of (i). (ii) is a direct result of \cite[Lem. 3.3 (ii)]{AAEN}. (iii) follows from the fact that $\HT$ preserves the norm $\| \, \|_{\E(\ree)}$ and the fact that $$({\HT}f)(x,t)=(\HT f(x,\cdot))(t)=(H_t0)(t)=0$$
 if $f \in \E(U \times \R), x \in \R^n \setminus U, t \in \R.$\end{proof}

\subsection{The space $\BMO$} The space $\BMO(\mathbb R^n)=\BMO(\mathbb R^n,\d x)$ is the space of real-valued functions with bounded mean oscillation with respect to $\d x$, i.e.,
$f\in \BMO(\mathbb R^n)$ if and only if
\begin{equation}\label{BMOw}
{\|f\|_{\BMO(\mathbb R^n)} :=} \sup_{Q \subset \R^n}\barint_{Q}\bigl |f-\barint_{Q} f\bigr|<\infty,
\end{equation}
where the supremum is taken with respect to all cubes $Q\subset\mathbb R^n$. We generalize the definition to $\BMO(\R^n;\mathbb{C}^k)$ by defining it to be the space of all the vector valued functions  $ (f_1,\cdot \cdot \cdot,f_k) \in \L^1(\mathbb{R}^n;\mathbb{C}^k)$, such that $\Re{f_i}, \Im{f_i} \in \BMO(\mathbb R^n)$ for all $1\leq i \leq k.$

\subsection{Assumptions on the coefficients}
\label{keysec}

The matrix-valued function  $A=A(x,t)=\{A_{i,j}(x,t)\}_{i,j=1}^{n}$ is assumed to have complex measurable entries $A_{i,j}$ which can be decomposed as in \eqref{eq1deg+h}, i.e., $A(x,t)=S(x,t)+D(x,t)$.
The  complex $n\times n$-dimensional matrix-valued function $S=S(x,t)=\{S_{i,j}(x,t)\}_{i,j=1}^{n}$ is assumed to satisfy the condition
\begin{equation}
\label{ellip}
 c_1|\xi|^2\leq \Re (S(x,t) \xi \cdot \cl{\xi}), \qquad
 |S(x,t)\xi\cdot\zeta|\leq c_2|\xi||\zeta|,
\end{equation}
for some $c_1, c_2\in (0,\infty)$ and for all $\xi,\zeta \in \IC^{n}$, $(x,t) \in \R^{n+1}$. Here, $u\cdot v=u_1\bar v_1+...+u_{n}\bar v_{n}$, where $\bar u$ denotes the complex conjugate of $u$ and $u\cdot \cl{v}$ is the standard inner product on $\IC^{n}$.

Concerning $D=D(x,t)=\{D_{i,j}(x,t)\}_{i,j=1}^{n}$ this function is assumed to be a real, measurable, $n\times n$-dimensional, {and} anti-symmetric matrix-valued function, and we write the regularity conditions impose on $D$ as
 \begin{equation}\label{bmocoeff}
D_{i,j}\in \L^\infty(\mathbb R,\BMO (\mathbb R^n,\d x)),\ \|D\|_{\L^\infty(\mathbb R,\BMO (\mathbb R^n,\d x))}\leq c_3,
\end{equation}
for some $c_3\in (0,\infty)$.

Note that the definition in \eqref{bmocoeff} may seem a bit informal and additional explanations are called for. Indeed, note that if $D(x,t)= G(t)$ for $x \in \R^n, t \in \R$, and for some locally integrable real anti-symmetric $n \times n$-dimensional matrix-valued function $G(t)$, then $D$  belongs to $\L^\infty(\mathbb R,\BMO (\mathbb R^n,\d x))$ and its norm in this space is zero. Hence, the anti-symmetric part $D$ need not be locally bounded in any $\L^p$ norm for $p>1$.  Instead, consider the modified matrix-valued function $D - \barint_{Q_0} D$ for an arbitrary but fixed cube $Q_0 \subset \R^n$. Then, by the John-Nirenberg's inequality, see \cite{John}, we obtain
\begin{equation}
\begin{aligned} \label{magicalineq}
    \bigl\| D - \barint_{Q_0} D \bigr\|^p_{\L^p(2^k Q_0 \times I)} & \leq \bigl\| D - \barint_{2^k Q_0} D \bigr\|^p_{\L^p(2^k Q_0 \times I)} + \sum_{i=1}^k \bigl\| \barint_{2^i Q_0} D - \barint_{2^{i-1} Q_0} D \bigr\|^p_{\L^p(2^k Q_0 \times I)} \\ &  \lesssim \bigl\| D - \barint_{2^k Q_0} D \bigr\|^p_{\L^p(2^k Q_0 \times I)} + \sum_{i=1}^k \bigl\| 2 \barint_{2^i Q_0} \bigl|D - \barint_{2^{i} Q_0} D\bigr| \, \bigr\|^p_{\L^p(2^k Q_0 \times I)} \\ &\lesssim_p k |2^k Q_0\times I|,
\end{aligned}
\end{equation}
for any $p\geq 1$, any positive integer $k$, and for any finite interval $I \subset \R$. Here, the implicit constant depends on $n$, $c_3$ and $p$.  Hence, $D - \barint_{Q_0} D$ belongs to $\L^p_{\loc}(\ree)$ for all $p \geq 1$.

Based on the above we can conclude that
\begin{align}\label{bmoo0}
\sup_{I \subset \mathbb{R}}\sup_{Q \subset \mathbb{R}^{n}}\biggl( \barint_{I} \barint_{Q} \bigl|D- \barint_Q D  \bigr|^p\, \d x\d t\biggr )^{1/p} \lesssim_p 1,
\end{align}
for every $p\geq 1$, where the first supremum is taken over all intervals $I \subset \mathbb{R}$, and the second supremum is taken over all cubes $ Q \subset \mathbb{R}^{n}$. It is in this form we will use \eqref{bmocoeff}, and strictly speaking this is the formal definition of what we summarize in \eqref{bmocoeff}. Consequently,
\begin{align}\label{bmoo1}
\sup_{I \subset \mathbb{R}}\sup_{Q \subset \mathbb{R}^{n}} \barint_{I} \barint_{Q} \bigl|A- \barint_Q A  \bigr|^2 \lesssim 1,
\end{align}
where the first supremum is taken over all intervals $I \subset \mathbb{R}$, and the second supremum is taken over all cubes $ Q \subset \mathbb{R}^{n}$.

Note also that if  $Q_0 \subset \R^n$ is an arbitrary but fixed cube, then $\barint_{Q_0} D$ is a real anti-symmetric matrix and $$\mbox{$\div_x ((\barint_{Q_0} D) \nabla_x u) = 0$ weakly for every $u \in \E(\ree)$.}$$
In fact,
\begin{align*}
    \iint_{\ree} \bigl(\barint_{Q_0} D\bigr) \nabla_x u \cdot \nabla_x \phi \, \d x \d t &= \frac 1 2 \sum_{i,j=1}^n\iint_{\mathbb R^{n+1}} \bigl(\barint_{Q_0} D\bigr)_{i,j}(x,t)(\partial_{x_i}u \, \cl{\partial_{x_j}\phi}-\partial_{x_j}u \, \cl{\partial_{x_i} \phi})\, \d x \d t  \\ &= -\frac 1 2 \sum_{i,j=1}^n\iint_{\mathbb R^{n+1}} \bigl(\barint_{Q_0} D\bigr)_{i,j}(x,t)(u \, \cl{\partial_{x_i} \partial_{x_j}\phi}-u \, \cl{\partial_{x_j} \partial_{x_i} \phi})\, \d x \d t = 0,
\end{align*}
for all $\phi \in C^{\infty}_0(\ree).$ Hence,
\begin{align}\label{normaled}
\partial_t -\div_{x} (A(x,t)\nabla_{x} ) =  \partial_t -\div_{x} \bigl(\bigl(A(x,t)- \bigl(\barint_{Q_0} D\bigr)(t)\bigr)\nabla_{x} \bigr),
\end{align}
 weakly in $\E(\ree)$. By replacing $D$ by $D - \barint_{Q_0} D$, we can therefore, without loss of generality and whenever we prefer, assume that
 \begin{align}\label{normal}
 \bigl(\barint_{Q_0} D\bigr )(t)= 0\mbox{ for  ${t \in \R}$ and }D \in \L^p_{\loc}(\ree)\mbox{ for all }p \geq 1.
 \end{align}

\subsection{Hardy spaces} The  Hardy space $\mathcal{H}^1_{\phi}(\mathbb{R}^{n})$ consists of all Lebesgue measurable functions $f:\mathbb{R}^{n} \to \mathbb{R}$ satisfying
$$
\|f\|_{\mathcal{H}_{\phi}^1(\mathbb{R}^{n})}:= \|\sup_{\epsilon>0} |\phi_{\epsilon} \ast f|\,\|_{\L^1(\mathbb{R}^{n})} < \infty,
$$
where $\phi_{\epsilon} = \epsilon^{-n} \phi(x/\epsilon)$ and $\phi \in C^{\infty}_0(B(0,1))$ is an approximation of the identity. Note that given two approximations of the identity, $\phi$ and $\psi$, the norms $\|\,\|_{\mathcal{H}^1_{\phi}(\mathbb{R}^n)}$ and $\|\,\|_{\mathcal{H}^1_{\psi}(\mathbb{R}^n)}$ are equivalent and the spaces $\mathcal{H}^1_{\phi}(\mathbb{R}^{n})$ and
$\mathcal{H}^1_{\psi}(\mathbb{R}^{n})$ are the same. Hence, by abuse of notation, we will denote all of them by $\mathcal{H}^1(\R^n)=\mathcal{H}^1(\R^n,\d x)$ having the norm $\|\,\|_{\mathcal{H}^1(\R^n)}$. For this fact, and for more on Hardy spaces we refer to \cite{fefferman-stein} and to chapter III in \cite{Stein}.

\subsection{Compensated compactness} The following classical compensated compactness estimates, see \cite{coifman1993compensated} and \cite{qian2019parabolic}, are essential when handling
the anti-symmetric part $D$ in $A$.

\begin{lem} Let $n\geq 2$ and let $f,g : \mathbb{R}^{n} \to \mathbb{R}$ be in $\H^1(\mathbb{R}^{n})$. Then,
\begin{align*}
\|\partial_{x_i} f\, \partial_{x_j} g - \partial_{x_i} g\, \partial_{x_j} f\|_{\mathcal{H}^1(\mathbb{R}^{n})} \lesssim& \, \|\nabla_x f\|_2 \|\nabla_x g\|_{2}, \\
\| f \xi  \cdot \nabla_x f  \|_{\mathcal{H}^1(\mathbb{R}^{n})} \lesssim& \,  |\xi|\,
\| f\|_2 \|\nabla_x f\|_2,
\end{align*}
for all $i,j \in \{1,...,n\}$ and for every $\xi \in \mathbb{R}^{n}$ and for  implicit constants depending only on $n$.
\label{p3.1}
\end{lem}

It is a standard result, see \cite{fefferman-stein}, that $\BMO(\mathbb{R}^n, \d x)$ is the dual of $\mathcal{H}^1(\mathbb{R}^n,\d x)$. Using this duality,
\begin{align}
    \iint_{\mathbb{R}^{n+1}} D_{ij}(x,t) f(x,t)  \d x \d t&\leq \|D\|_{\L^{\infty}(\mathbb{R},\BMO(\mathbb{R}^n,\d x))} \int_{\mathbb R}  \|f(\cdot,t)\|_{\mathcal{H}^1(\mathbb{R}^n, \d x)}\, \d t\notag\\
    &\lesssim \int_{\mathbb R}  \|f(\cdot,t)\|_{\mathcal{H}^1(\mathbb{R}^n, \d x)}\, \d t,
    \label{eqbmo}
\end{align} for every $1\leq i,j \leq n$ and every function $f:\mathbb{R}^{n+1} \to \mathbb{R}$ in  $\L^1_{\mbox{loc}}(\mathbb R,\mathcal{H}^1(\mathbb{R}^n,\d x))$. In general, we will apply Lemma \ref{p3.1} in conjunction with \eqref{eqbmo}. Note also that \begin{align*}
&\iint_{\mathbb R^{n+1}}D(x,t)\gradx u \cdot \cl{\gradx v}\, \d x \d t=\frac 1 2 \sum_{i,j=1}^n\iint_{\mathbb R^{n+1}}D_{i,j}(x,t)(\partial_{x_i}u \, \cl{\partial_{x_j}v}-\partial_{x_j}u \, \cl{\partial_{x_i} v})\, \d x \d t,
\end{align*}
by the anti-symmetry property of $D$, whenever $u,v : \mathbb{R}^{n+1} \to \mathbb{R}$ are in $C_0^\infty(\mathbb{R}^{n+1})$.

\section{Maximal functions and Littlewood-Paley theory}\label{max}

 We introduce the maximal operators in the individual variables
\begin{align*}
	\mathcal{M}^{(1)}(g_1)(x) &:= \sup_{r>0 }{\barint_{Q_r(x)} |g_1|\, \d y},\\
	\mathcal{M}^{(2)}(g_2)(t) &:= \sup_{r>0 }{\barint_{I_r(t)} |g_2|\, \d s},
\end{align*}
for all locally integrable functions $g_1$ and $g_2$ on $\R^{n}$ and $\R$, respectively. Recall that  $\Delta_r(x,t):= Q_r(x)\times I_r(t)\subset\mathbb R^{n+1}$ is the  parabolic cube centered at $(x,t)$ and of size $r$.

Throughout the paper,  $\P\in C_0^\infty(\ree)$ will be a fixed real-valued function  in product form
$$\P(x,t)=\P^{(1)}(x)\P^{(2)}(t),$$
where $\P^{(1)}\in C_0^\infty(\mathbb R^n)$ and $\P^{(2)}\in C_0^\infty(\mathbb R)$ both are radial and have integral $1$. For $x \in \R^n, t \in \R$, we set
\begin{align*}
    &\P_\lambda^{(1)}(x) := \lambda^{-n} \P^{(1)}(x/\lambda), \\
    &\P_\lambda^{(2)}(t) := \lambda^{-2} \P^{(2)}(t/\lambda^2), \\
    & \P_\lambda(x,t):=\P_\lambda^{(1)}(x)\P_\lambda^{(2)}(t)=\lambda^{-n-2}\P^{(1)}(x/\lambda) \P^{(2)}(t/\lambda^2),
\end{align*}
 whenever $\lambda>0$.  We will need the additional cancellation property
\begin{align} \label{momentum}
   \int_{\R^n}  x_{i} \, \P^{(1)}(x) \, \d x =0,
\end{align}
for all integers $1 \leq i \leq n$. This property is a result of $\P^{(1)}$ being radial and $x_{i} \, \P^{(1)}(x)$ being an odd function of $x_i$. With a slight abuse of notation, we let $\P_\lambda^{(1)}$, $\P_\lambda^{(2)}$, and $\P_\lambda$ also denote the associated convolution operators
\begin{align*}
\P_\lambda^{(1)}f(x,t)&:=\P_\lambda^{(1)}\ast f(x,t):=\iint_{\mathbb R^{n}}\P_\lambda^{(1)}(x-y)f(y,t)\, \d y,\\
\P_\lambda^{(2)}f(x,t)&:=\P_\lambda^{(2)}\ast f(x,t):=\iint_{\mathbb R}\P_\lambda^{(2)}(t-s)f(x,s)\, \d s,\\
\P_\lambda f(x,t)&:=\P_\lambda\ast f(x,t):=\iint_{\mathbb R^{n+1}}\P_\lambda(x-y,t-s)f(y,s)\, \d y\d s.
\end{align*}
We note that
\begin{equation}
\begin{aligned}\label{dda}
|\P^{(1)}_\lambda f(x,t)| & \leq  \mathcal{M}^{(1)}(f(\cdot,t))(x),\\ |\P^{(2)}_\lambda f(x,t)| & \leq  \mathcal{M}^{(2)}(f(x,\cdot))(t), \\ |\P_\lambda f(x,t)| & \leq  \mathcal{M}^{(1)}(\mathcal{M}^{(2)}(f(x,\cdot))(t))(x),
\end{aligned}
\end{equation}
almost everywhere for $f \in \L^1_{\loc}(\ree)$, see \cite[Sec.~II.2.1]{Stein}.  The following lemma is a result of the Hardy-Littlewood inequality and \eqref{dda}.
\begin{lem} \label{lem:approx}
The operator $\P_\lambda$ is bounded on $\L^2(\mathbb R^{n+1})$, and $\|\P_\lambda\|_{2\to 2}\lesssim 1$. Furthermore,  if $f\in \L^2(\mathbb R^{n+1})$, then $\P_\lambda f$ is smooth and $\P_\lambda f\to  f$  in $\L^2(\mathbb R^{n+1})$ as $\lambda\to 0$.
\end{lem}

In the context of square functions and the Littlewood-Paley theory, we will use the notation
      \begin{eqnarray}\label{tnorm}
      |||\cdot|||_{2}:=\biggl (\int_0^\infty\iint_{\mathbb R^{n+1}}|\cdot|^2\, \frac{\d x \d t \d\lambda}\lambda\biggr )^{\frac{1}{2}}.
      \end{eqnarray}
      Lemma \ref{little1}, Lemma \ref{little2}, and Lemma \ref{little3} below will be used frequently in the forthcoming sections.
\begin{lem}\label{little1} For all $f\in \L^2(\mathbb R^{n+1})$, we have
\begin{eqnarray*}\label{li-}|||\lambda\nabla_x \P_\lambda f|||_{2}+|||\lambda^2\partial_t\P_\lambda f|||_{2}+|||\lambda \dhalf \P_\lambda f|||_{2}\lesssim \|f\|_{2}.
\end{eqnarray*}
\end{lem}
\begin{proof} For the convenience of the reader, we here provide the proof of the estimate for one of the terms,  $|||\lambda\nabla_x \P_\lambda f|||_2$. The proof uses  only the Fourier transform. We also refer to Lemma 5.1 in \cite{AAEN} for similar derivations. Let $f \in \L^2(\mathbb R^{n+1})$ and let $\P_{\lambda}= \P_\lambda^{(1)}\P_\lambda^{(2)}$ be as before. Then, using the Fourier transform and Plancherel's theorem in $x$, Hardy-Littlewood inequality, and Lemma \ref{lem:approx}, we obtain
\begin{align*}
|||\lambda\nabla_x \P_\lambda f|||^2_{2} &= \int_{0}^{\infty} \iint_{\ree} |\lambda\nabla_x \P_\lambda f|^2 \, \frac{\d x \d t \d \lambda}{\lambda} \\  &= \int_{0}^{\infty} \iint_{\ree} |\P_\lambda^{(2)}\lambda\nabla_x \P_\lambda^{(1)} f|^2 \, \frac{\d x \d t \d \lambda}{\lambda} \\ & \lesssim \int_{0}^{\infty} \iint_{\ree} |\lambda\nabla_x \P_\lambda^{(1)} f|^2 \, \frac{\d x \d t \d \lambda}{\lambda} \\ &=\iint_{\ree}  \int_{0}^{\infty}  |\lambda \chi \widehat{\P_\lambda^{(1)}}(\lambda \chi) \hat{f}(\chi,t)|^2 \, \frac{\d \chi \d t \d \lambda}{\lambda}
\\&=\iint_{\ree} |\hat{f}( \chi,t)|^2  \int_{0}^{\infty}  |\lambda \chi \widehat{\P_\lambda^{(1)}}(\lambda \chi)|^2 \, \frac{\d  \lambda}{\lambda} \, \d \chi \d t
\\& \lesssim  ||f||^2_{2}.
\end{align*}
Here we have used, in the last inequality, that $\widehat{\P_\lambda^{(1)}}$ is a radial Schwartz function, which proves the integral over $\lambda$ is finite and independent of $\chi$. \end{proof}

Recall in the notation ${\mathbb D}= (\nabla_x,\dhalf)$ introduced in \eqref{pargradient}.

\begin{lem} \label{little2} For all $f \in \E(\ree)$, we have
\begin{eqnarray*}
|||\lambda^{-1}(I-\P_\lambda) f|||_{2}\lesssim \|\mathbb D f\|_{2}.
\end{eqnarray*}
\end{lem}
\begin{proof} This lemma is proved in \cite{N1}, but we here include a proof relying  only on the Fourier transform for the convenience of the reader. Let $f \in \E(\ree)$. We first note that
\begin{eqnarray}\label{standard}
|||\lambda^{-1}(I-\P_\lambda^{(1)}) f|||_{2}+|||\lambda^{-1}(I-\P_\lambda^{(2)}) f|||_{2}\lesssim ||{\mathbb D} f||_{2},
\end{eqnarray}
for all $f\in \E(\ree)$. Indeed, using the Fourier transform and  Plancherel's theorem in $t$, we obtain
\begin{align*}
    |||\lambda^{-1}(I-\P_\lambda^{(2)}) f|||^2_{2} &= \int_{0}^{\infty} \iint_{\ree} \big|\lambda^{-1}(I-\widehat{\P_\lambda^{(2)}}(\lambda^2 \tau)) \hat{f}\big|^2(x,\tau)  \, \frac{\d x \d \tau \d \lambda}{\lambda} \\ &=  \iint_{\ree} \int_{0}^{\infty} \big|(\tau^{1/2} \lambda)^{-1}(I-\widehat{\P_\lambda^{(2)}}(\lambda^2 \tau)) \widehat{\dhalf f}\big|^2(x,\tau)  \, \frac{\d x \d \tau \d \lambda}{\lambda} \\ &\lesssim \| \dhalf f \|_2^2.
\end{align*}
Here we have used that $\widehat{\P_\lambda^{(2)}}$ is a radial Schwartz function, and classical Littlewood-Paley theory in the variable $t$, in the last inequality. A similar argument in $x$ implies $|||\lambda^{-1}(I-\P_\lambda^{(1)}) f|||_{2} \lesssim \|\nabla_x f\|_2$,  and this concludes the proof of \eqref{standard}. {Since}
\begin{align*}
(I-\P_\lambda)=\P_\lambda^{(2)}(I-\P_\lambda^{(1)})+(I-\P_\lambda^{(2)}),
\end{align*}
 we can use \eqref{dda} and \eqref{standard} {to} deduce
\begin{align*}
|||\lambda^{-1}(I-\P_\lambda) f|||_{2}&\leq |||\lambda^{-1}\P_\lambda^{(2)}(I-\P_\lambda^{(1)}) f|||_{2}+|||\lambda^{-1}(I-\P_\lambda^{(2)}) f|||_{2}\notag\\
&\lesssim |||\lambda^{-1}(I-\P_\lambda^{(1)}) f|||_{2}+|||\lambda^{-1}(I-\P_\lambda^{(2)})f |||_{2}\lesssim ||{\mathbb D} f||_{2}.
\end{align*}
 \end{proof}

In the following, $\mathcal{A}_\lambda$ denotes the dyadic averaging operator, i.e., if $\hat \Delta_\lambda(x,t)$ is the dyadic parabolic cube with size $\lambda \leq \ell(\hat \Delta_\lambda(x,t)) < 2 \lambda$, containing $(x,t)$, then
      \begin{eqnarray}\label{dy}
      \mathcal{A}_\lambda f(x,t):=\bariint_{\hat \Delta_\lambda(x,t)}f.
      \end{eqnarray}
     {In other words}, $\mathcal{A}_\lambda f(x,t)$ is the average of $f$ over $\hat \Delta_\lambda(x,t)$.

\begin{lem} \label{little3} For all $f\in \L^2(\mathbb R^{n+1})$, we have
\begin{eqnarray*}
|||(\mathcal{A}_\lambda -\P_\lambda)f|||_{2}\lesssim \|f\|_{2}.
\end{eqnarray*}
\end{lem}
\begin{proof}  We refer to \cite{N1} and \cite{CNS} for the proof of the lemma.
\end{proof}

\section{The sesquilinear form and maximal accretivity} \label{form}

 Recall that $\HT$ is the Hilbert transform with respect to the $t$-variable. We factorize $\partial_t = \dhalf \HT \dhalf$ and define $\cH$ as an operator from $\E(\ree)$ into its (anti)-dual $\E(\ree)^*$ via
\begin{align}
\label{hidden coercivity}
\langle\cH u,v\rangle_2 := \iint_{\ree} A(x,t) \gradx u \cdot \cl{\gradx v} +\HT \dhalf u \cdot \cl{\dhalf v} \, \d x\d t.
\end{align}
In fact, assuming \eqref{ellip} and \eqref{bmocoeff}, $\cH$ is a bounded operator $\E(\ree)\to\E(\ree)^*$. To see this, we  note that
\begin{align}\label{obs1}
&\iint_{\mathbb R^{n+1}}D(x,t)\gradx u \cdot \cl{\gradx v}\, \d x \d t=\frac 1 2 \sum_{i,j=1}^n\iint_{\mathbb R^{n+1}}D_{i,j}(x,t)(\partial_{x_i}u \, \cl{\partial_{x_j}v}-\partial_{x_j}u \, \cl{\partial_{x_i} v})\, \d x \d t,
\end{align}
by the anti-symmetry property of $D$. Using \eqref{bmocoeff} and the duality between $\BMO(\mathbb R^n,\d x)$ and the Hardy space $\mathcal{H}^1(\R^n,\d x)$, see \eqref{eqbmo},
\begin{align}\label{obs1+}
\biggl |\iint_{\mathbb R^{n+1}}D(x,t)\gradx u \cdot \cl{\gradx v}\, \d x \d t\biggr |&\lesssim \int_{\mathbb R}\|\partial_{x_i}u(\cdot,t)\, \cl{\partial_{x_j}v(\cdot,t)}-\partial_{x_j}u(\cdot,t)\, \cl{\partial_{x_i} v(\cdot,t)}\|_{\mathcal{H}^1(\R^n,\d x)}\, \d t.
\end{align}
Hence, applying Lemma \ref{p3.1} and the Cauchy-Schwarz {inequality}, we conclude that
\begin{align}\label{obs1+gg}
\biggl |\iint_{\mathbb R^{n+1}}D(x,t)\gradx u \cdot \cl{\gradx v}\, \d x \d t\biggr |&\lesssim\int_{\mathbb R}\|\gradx u(\cdot,t)\|_2\|\gradx v(\cdot,t)\|_2\, \d t\lesssim\|\gradx u\|_2\|\gradx v\|_2.
\end{align}
The deductions in \eqref{obs1}-\eqref{obs1+gg} hold for $u,v\in C_0^\infty(\mathbb R^{n+1})$, and hence whenever  $u,v\in \E(\ree)$ as
$C_0^\infty(\mathbb R^{n+1})$ is dense in $\E(\ree)$, see  Lemma 3.3 in \cite{AAEN}. Furthermore, an analogous construction applies to $\cH^*$, the adjoint or dual of $\cH$, which formally coincides with the operator
\begin{eqnarray}\label{eq1deg+adj}
-\partial_t -\div_{x} (  A^\ast(x,t)\nabla_{x} ),\ (x,t)\in \mathbb R^{n+1},
 \end{eqnarray}
 where $A^\ast$ is the (Hermitian) transpose of $A$.

 Assume  \eqref{ellip} and \eqref{bmocoeff}. Throughout the paper we consider the part of $\cH$ in $\L^2(\ree)$, defined by
\begin{align}\label{domain}
\dom(\cH) := \{u \in \E(\ree) : \cH u \in \L^2(\ree) \}.
\end{align}
Recall that by definition this means that if $u \in \E(\ree)$, then
$u\in \dom(\cH)$ if and only if there exists a constant c such that,
\begin{align*}
|\langle\cH u,v\rangle_2|= \biggl |\iint_{\ree} A(x,t) \gradx u \cdot \cl{\gradx v} +\HT\dhalf u \cdot \cl{\dhalf v} \, \d x\d t\biggr |\leq c \|v\|_2,
\end{align*}
for all $v \in \E(\ree)$.  The part of $\cH$ in $\L^2(\ree)$, with maximal domain $\dom(\cH)$, is said to be maximal accretive if for every $\sigma \in \IC$ with $\Re \sigma > 0$, the operator $\sigma + \cH$ is invertible and $\|(\sigma + \cH)^{-1}\| \leq (\Re \sigma)^{-1}$ holds.

\subsection{Maximal accretivity} The operator $\cH$ was introduced as  a bounded operator from $\E(\ree)$ into its (anti)-dual $\E(\ree)^*$ via the sesquilinear form in
\eqref{hidden coercivity}. For short, we identify $\cH=\partial_t -\div_{x} ( A(x,t)\nabla_{x} )$ and  $\cH^\ast=-\partial_t -\div_{x} ( A^\ast(x,t)\nabla_{x} )$. In the following lemma we establish hidden coercivity. In the case $D=0$ this was observed in \cite{AAEN,N1}.

\begin{lem}
\label{Lem: WP on R} Let $\sigma \in \IC$ with $\Re \sigma > 0$. The following assertions hold.
\begin{enumerate}
 \item For each $f \in \E(\ree)^*$ there exists a unique $u \in \E(\ree) $ such that $(\sigma + \cH) u = f$. Moreover,
  \begin{align*}
  \|u\|_{\E(\ree)} \leq \sqrt{2} \max \Big\{\frac{c_2+c_3 + 1}{c_1}, \frac{|\Im \sigma| + 1}{\Re \sigma} \Big\} \|f\|_{\E(\ree) ^*}.
  \end{align*}
\item If $f \in \L^2(\mathbb R^{n+1})$, then
  \begin{align*}
  \|u\|_{2} \leq \frac{1}{\Re \sigma}\|f\|_{2}.
  \end{align*}
  In particular, the part of $\cH$ in $\L^2(\ree)$ is maximal accretive with domain $\dom(\cH)$ as in \eqref{domain}.
\item The space $\dom(\cH)$ is dense in $\E(\ree)$.
\item All the above results hold if we replace $\cH$ by $\cH^*.$
\end{enumerate}
\end{lem}
\begin{proof}
 We define the sesquilinear form
$ B_{\delta,\sigma}: \E(\ree) \times\E(\ree) \to \IC$ by
\begin{align*}
 B_{\delta,\sigma}(u,v) := \iint_{\R^{n+1}} A\nabla_x u \cdot \cl{\nabla_x(1+\delta \HT) v} + \HT \dhalf u \, \cl{\dhalf (1+\delta \HT) v}+\sigma u\, \cl{(1+\delta \HT) v}\, \d x\d t,
\end{align*}
where ${0 <\delta< 1}$ is a  real number yet to be chosen.  First, as in \eqref{obs1}-\eqref{obs1+gg}, we deduce
\begin{align}
\label{eq1}
| B_{\delta,\sigma}(u,v)|&{\lesssim} \|\nabla_x u\|_{2}\|\nabla_x v\|_{2}+\| \dhalf u\|_{2}\| \dhalf v\|_{2}+|\sigma|\|u\|_{2}\|v\|_{2}\notag\\
&{\lesssim} (|\sigma|+1)  \|u\|_{\E(\ree)} \|v\|_{\E(\ree)}.
\end{align}
Second, using that
\begin{align}\label{obs2}
\Re D(x,t)\xi\cdot\bar\xi=0,\ (x,t)\in\mathbb R^{n+1},\ \xi\in \mathbb C^n,
\end{align}
by the anti-symmetry property of $D$, we see that
\begin{align}
\label{eq2}
 \Re  B_{\delta,\sigma}(v,v)
\geq \delta \|\dhalf v\|_{2}^2 + (c_1 - c_2\delta-c_3\delta )\|\nabla_x v\|_{2}^2 + (\Re \sigma - \delta |\Im \sigma|) \|v\|_{2}^2,
\end{align}
by \eqref{obs1} and \eqref{obs1+gg}. Choosing $\delta$ such that the factors in front of the second and third term in the last display are no less than $\delta$, for instance choosing
\begin{align*}
    \delta :=  \min \Big \{\frac{c_1}{c_2+c_3 + 1}, \frac{\Re \sigma}{|\Im \sigma|+1} \Big \},
\end{align*}
we obtain the coercivity estimate
\begin{align*}
 \Re B_{\delta,\sigma}(v,v) \geq  \min \Big \{\frac{c_1}{c_2+c_3 + 1}, \frac{\Re \sigma}{|\Im \sigma|+1} \Big \} \|v\|_{\E(\ree)}^2 \qquad (v \in \E(\ree)).
\end{align*}
The Lax-Milgram lemma now yields, for each $f \in {\E(\ree)^*}$,  a unique $u \in \E(\ree)$ satisfying the bound stated in (i), and such that
\begin{align*}
 B_{\delta,\sigma}(u,v) = f((1+\delta \HT)v) \qquad (v \in \E(\ree)).
\end{align*}
 Plancherel's theorem yields that $1+\delta \HT$ is an isomorphism on $\E(\ree)$ for all $\delta\in\mathbb R$ (note that the symbol for $1+\delta \HT$ is $(1+i\sgn(\tau))$ . Thus,
\begin{align*}
 \iint_{\R^{n+1}} A\nabla_x u \cdot \cl{\nabla_x v} + \HT \dhalf u \cdot \cl{\dhalf  v}+\sigma  u\,\cl{v}\, \d x\d t=\iint_{\R^{n+1}} f \,\cl{v}\, \d x\d t\, \qquad (v \in \E(\ree)),
\end{align*}
that is, $(\sigma + \cH) u  = f$. This completes the proof of $\mathrm{(i)}$.

To prove (ii), we first note that from (i) we can conclude that $\sigma + \cH: \dom(\cH) \to \L^2(\mathbb R^{n+1})$ is one-to-one. To check the resolvent estimate required for maximal accretivity, let $f \in \L^2(\mathbb R^{n+1})$. Recall that $A=S+D$. Then, by accretivity of $S$, \eqref{obs2}, and skew-adjointness of the Hilbert transform
\begin{align*}
 \Re \sigma \|u\|_{2}^2
 &\leq \Re \iint_{\R^{n+1}} A\nabla_x u \cdot \cl{\nabla_x u} + \HT \dhalf u \cdot \cl{\dhalf  u}+\sigma u\cdot\cl{u}\, \d x\d t \\
 &= \Re \iint_{\R^{n+1}} f \cdot\cl{u}\,\d x\d t \leq \|f\|_{2} \|u\|_{2},
\end{align*}
and this completes the proof of $\mathrm{(ii)}$.

To prove (iii) we use the argument in Proposition 4.2 in \cite{AAEN}. Note that, by \eqref{eq1} and \eqref{eq2}, the sesquilinear form $ B_{\delta,1}(\cdot \,,\, \cdot ) $ is bounded and accretive on $\E(\ree)$ for  $\delta>0$ small enough. Now, if $\dom(\cH)$ is not dense in $\E(\ree)$, there exists a nontrivial element $v \in \E(\ree)$ such that $\langle u,v \rangle_{\E(\ree)} =0 $ for every $u \in \dom(\cH)$. By Lax-Milgram theorem, there exists a nontrivial element $\tilde{v} \in \E(\ree)$ such that $ B_{\delta,1}(u,\tilde{v}) = \langle u,v \rangle_{\E(\ree)}$ for every $u \in \E(\ree)$. Hence,
$$
0 = \langle (1+\cH) u, (1+\delta \HT) \tilde{v} \rangle\quad\mbox{for every $u \in \dom(\cH)$.}
$$
Using that $(1+\cH):\dom(\cH)\to\L^2(\ree)$ is a bijection, we conclude that $(1+\delta \HT) \tilde{v}=0$ and consequently that $\tilde{v}=0 = v$. This is a contradiction, and hence $\dom(\cH)$ is  dense in $\E(\ree)$.

Finally, the proof of (iv) follows from the same argument as for $\cH,$ but with $-\partial_t$ instead of $\partial_t$ and $A$ replaced by  $A^*= S^*-D$.\end{proof}

\subsection{Bounds for the resolvents} We introduce the resolvent operators
     \begin{eqnarray}\label{resolvents}
     \mathcal{E}_\lambda:=(I+\lambda^2\cH)^{-1},\ \mathcal{E}_\lambda^\ast:=(I+\lambda^2\cH^\ast)^{-1}.
     \end{eqnarray}
 \begin{lem}\label{le8-} For any $\lambda\in \R$, the following estimates hold
       \begin{align*}
         \mathrm{(i)}& \quad  \|\mathcal{E}_\lambda f\|_{2}+\|\lambda \mathbb D\mathcal{E}_\lambda f\|_{2}\lesssim \| f\|_{2},\\\
        \mathrm{(ii)}& \quad  \|\lambda \mathcal{E}_\lambda \dhalf f \|_{2} + \|\lambda^2 \mathbb D\mathcal{E}_\lambda \dhalf f \|_{2}\lesssim \|{f}\|_{2},\\
        \mathrm{(iii)}& \quad \|\lambda \mathcal{E}_\lambda \div_x({\bf f})\|_{2} + \|\lambda^2\mathbb D \mathcal{E}_\lambda \div_x({\bf f}) \|_{2}\lesssim \|{\bf f}\|_{2},
       \end{align*}
        for all $f\in \L^2(\mathbb R^{n+1})$, ${\bf f}\in \L^2(\mathbb R^{n+1};\mathbb C^{n})$. The {above inequalities} also hold with $\mathcal{E}_{\lambda}$ replaced by $\mathcal{E}^*_{\lambda}.$
    \end{lem}
\begin{proof} The lemma is an immediate consequences of Lemma \ref{Lem: WP on R} and its proof.\end{proof}

\begin{rem} \label{rem3} Let $U \subset\mathbb R^{n}$ be an open and bounded set. Repeating the argument used in the proof of Lemma \ref{Lem: WP on R}, we can consider the sesquilinear form $B_{\delta,\sigma}$ on $\E(U \times \R) \times \E(U \times \R)$ and prove the coercivity and boundedness for $B_{\delta,\sigma}$ whenever $\delta>0$ is small enough. Hence, by the Lax-Milgram lemma and the fact that $1+\delta \HT$ is an isomorphism on $\E(U \times \R)$, see Lemma \ref{space} (iii), we can conclude that there exist $u, \tilde{u} \in \E(U \times \R)$, for $f \in \L^2(U \times \R),\, \textbf{f} \in \L^2(U \times \R;\mathbb{C}^n)$ given,
such that
$$\begin{aligned}
\iint_{\ree} \lambda^2 A \nabla_x u \cdot \overline{\nabla_x v} + \lambda^2 \HT \dhalf u \cdot \overline{ \dhalf v} + u\, \overline{v} \, \d x \d t &= \iint_{\ree} f\, \overline{v} \, \d x \d t,\\
\iint_{\ree} \lambda^2 A \nabla_x \tilde{u} \cdot \overline{\nabla_x v} + \lambda^2 \HT \dhalf \tilde{u} \cdot \overline{ \dhalf v} + \tilde{u} \, \overline{v} \d x \d t &=- \iint_{\ree} \textbf{f} \cdot \overline{\nabla_x v} \, \d x \d t,
\end{aligned}
$$
for all $v \in \E(U \times \R)$. Furthermore,
 \begin{align*}
        \|u\|_{2}+\|\lambda \mathbb Du\|_{2}&\lesssim \| f\|_{2}\quad\mbox{and}\quad \|\lambda \tilde{u}\|_{2} + \|\lambda^2\mathbb D \tilde{u} \|_{2}\lesssim \|{\bf f}\|_{2}.
       \end{align*}
\end{rem}

\subsection{Reverse H{\"o}lder inequalities for the gradient of resolvents}
We will need the following reverse  H{\"o}lder inequalities for the gradient of resolvents.

\begin{lem}
\label{thm:reverseholdergradient}
There exists  $p>2$, depending only on the structural constants, such that
\begin{equation}
\label{eq:inverseholdergradient-}
    \begin{aligned}
    & \| \nabla_x \mathcal{E}^*_{\lambda} f\|_{\L^p(\Delta)} \lesssim |\Delta|^{\frac{1}{p}- \frac{1}{2}}\|\nabla_x \mathcal{E}^*_{\lambda} f\|_{\L^2(10\Delta)},\\
     &   \| \nabla_x \mathcal{E}_{\lambda} f\|_{\L^p(\Delta)} \lesssim |\Delta|^{\frac{1}{p}- \frac{1}{2}}\|\nabla_x \mathcal{E}_{\lambda} f\|_{\L^2(10\Delta)},
    \end{aligned}
    \end{equation}
    for all parabolic cubes $\Delta \subset \mathbb{R}^{n+1}$ and for all  $f \in \L^2(\R^{n+1})$ such that $f\equiv 0$ on $100\Delta$.
\end{lem}

\begin{proof} We will only discuss the proof in the case of $\mathcal{E}_{\lambda} $ and to prove this we can argue as in the proof of Theorem 2.2. in \cite{seregin2012divergence}. We are considering the equation
    \begin{align}\label{eka-}
        (1+\lambda^2(\partial_t - \div_x A \nabla_x)) u = f,
    \end{align}
    where $u := \mathcal{E}_{\lambda} f$. Note that we are aiming for a local regularity estimate for $\nabla_xu=\nabla_x\mathcal{E}_{\lambda} f$ in $10\Delta$ and that by assumption $f\equiv 0$ on $100\Delta$. Hence, we are considering the homogeneous equation
    \begin{align}\label{eka}
        (1+\lambda^2(\partial_t - \div_x A \nabla_x)) u = 0\quad\mbox{weakly in}\quad 100\Delta.
    \end{align}
The main difference with our situation compared to \cite{seregin2012divergence} is that in  \cite{seregin2012divergence} the authors consider the equation
$(\partial_t - \div_x A \nabla_x)u=0$ and they only consider real coefficients. To outline the reduction to Theorem 2.2. in \cite{seregin2012divergence}, let
    $v(x,t) := u\left(\lambda x,\lambda^2 t\right)$  and note that if $u$ solves \eqref{eka}, then $v$ solves
     \begin{align}\label{ekka}
        (1+\partial_t - \div_x \tilde{A} \nabla_x) v =0\quad\mbox{weakly in}\quad 100\tilde\Delta,\ \tilde\Delta:=\lambda^{-1}\Delta
    \end{align}
     and where $\tilde{A}(x,t) =A(\lambda x, \lambda^2 t).$  Assume that we have been able to prove that
    \begin{equation}
\label{eq:inverseholdergradient-+}
    \begin{aligned}
     &   \| \nabla_xv\|_{\L^p(\tilde\Delta)} \lesssim |\tilde\Delta|^{\frac{1}{p}- \frac{1}{2}}\|\nabla_x v\|_{\L^2(10\tilde\Delta)}.
    \end{aligned}
    \end{equation}
    Then, using \eqref{eq:inverseholdergradient-+} we see that
    \begin{align*}
\| \nabla_x \mathcal{E}_{\lambda} f\|_{\L^p(\Delta)}
 &=    \| \nabla_x u\|_{\L^p(\Delta)}\\
 &=  \left(\iint_{\Delta} \left| \nabla_x  \left(v\left(\lambda^{-1} x, \lambda^{-2} t\right) \right) \right|^p\, \d x \d t \right)^{\frac{
1}{p}}=\lambda^{\frac{n+2}{p}}\lambda^{-1}
\left(\iint_{\tilde \Delta} \left| \nabla_x  v\left(x,t\right)  \right|^p\, \d x \d t \right)^{\frac{
1}{p}}\\
&\lesssim \lambda^{\frac{n+2}{p}} \lambda^{-1}\left|\tilde \Delta\right|^{\frac{1}{p}- \frac{1}{2}}
\left(\iint_{10\tilde \Delta} \left| \nabla_x  v\left(x,t\right)  \right|^2 \, \d x \d t \right)^{\frac{
1}{2}}= |\Delta|^{\frac{1}{p}- \frac{1}{2}}  \left(\iint_{10\Delta} |   \nabla_x u(x, t)|^2 \, \d x \d t \right)^{\frac{
1}{2}}
\\
  & \lesssim   \left|\Delta\right|^{\frac{1}{p}- \frac{1}{2}}\|\nabla_x u\|_{\L^2(10 \Delta)}=\left|\Delta\right|^{\frac{1}{p}- \frac{1}{2}}\|\nabla_x \mathcal{E}_{\lambda} f\|_{\L^2(10 \Delta)}.
\end{align*}
Hence, we only have to prove  \eqref{eq:inverseholdergradient-+}, i.e,  we have to prove that if $v$ satisfies \eqref{ekka}, then  there exists  $p>2$, depending only on the structural constants, such that
\begin{equation}
\label{eq:inverseholdergradient-+ha}
    \begin{aligned}
     &   \biggl (\bariint_{\tilde\Delta}| \nabla_xv|^p\, \d x\d t\biggr )^{1/p}\lesssim \biggl (\bariint_{10\tilde\Delta}| \nabla_xv|^2\, \d x\d t\biggr )^{1/2}.
    \end{aligned}
    \end{equation}
    To establish this result following the proof of Theorem 2.2 in \cite{seregin2012divergence}, we note that the argument relies on standard tools, including Caccioppoli (energy) estimates for parabolic equations, the Poincaré-Sobolev inequality, and the Sobolev inequality in spatial variables. The consideration of \( A = S + D \), where \( S \) is complex but coercive, is permissible within the framework of the argument. Additionally, the inclusion of \( v \) with the term \( (\partial_t - \div_x \tilde{A} \nabla_x) v \) in the equation does not affect the validity of the reasoning. For further details, we refer the reader to \cite{seregin2012divergence}, and for additional context, see also \cite{GS}. This completes the proof.\end{proof}

    \begin{rem}\label{conseq}
A consequence of Lemma \ref{thm:reverseholdergradient} is the following: consider a parabolic cube \( \Delta \subset \mathbb{R}^{n+1} \) and a function \( g \in \L^2(\mathbb{R}^{n+1}) \) with \( g \equiv 0 \) outside of \( \Delta \). Then, for all \( k \geq 11 \), we have
\begin{equation}
\label{eq:inverseholdergradient}
\begin{aligned}
    & \| \nabla_x \mathcal{E}^*_{\lambda} g \|_{\L^p(2^{k+1} \Delta \setminus 2^k \Delta)}
    \lesssim |2^k \Delta|^{\frac{1}{p} - \frac{1}{2}} \| \nabla_x \mathcal{E}^*_{\lambda} g \|_{\L^2(2^{k+5} \Delta \setminus 2^{k-4} \Delta)}, \\
    & \| \nabla_x \mathcal{E}_{\lambda} g \|_{\L^p(2^{k+1} \Delta \setminus 2^k \Delta)}
    \lesssim |2^k \Delta|^{\frac{1}{p} - \frac{1}{2}} \| \nabla_x \mathcal{E}_{\lambda} g \|_{\L^2(2^{k+5} \Delta \setminus 2^{k-4} \Delta)}.
\end{aligned}
\end{equation}
\end{rem}

\section{Off-diagonal estimates}\label{off}
 We will need the following lemma.

\begin{lem}
There exists a constant $\theta \in (0,1)$, which only depends on the structural constants, such that the following holds. If  $\chi \in \mathbb{R}^n$ satisfies $|\chi|\leq\theta$, then
    \begin{align*}
   \mathrm{(i)}& \quad  \|e^{\big(\frac{x \cdot \chi}{\lambda} \big)} \mathcal{E}_\lambda f \|_{2}\lesssim  \|e^{\big(\frac{x \cdot \chi}{\lambda} \big)}  f \|_{2},\\
\mathrm{(ii)}& \quad  \| \lambda e^{\big(\frac{x \cdot \chi}{\lambda} \big)} \nabla_x  \mathcal{E}_\lambda f \|_{2} \lesssim \|e^{\big(\frac{x \cdot \chi}{\lambda} \big)}  f  \|_{2}, \\
\mathrm{(iii)}& \quad \| \lambda e^{\big(\frac{x \cdot \chi}{\lambda} \big)} \mathcal{E}_\lambda  \,\div_x \textbf{f}\,\|_{2} \lesssim \|e^{\big(\frac{x \cdot \chi}{\lambda} \big)}  \textbf{f}\,\|_{2},
    \end{align*}
for all $f \in \L^2(\ree)$ and $\textbf{f} \in \L^2(\ree;\mathbb{C}^n).$ The estimates also hold with $\mathcal{E}_{\lambda}$ replaced by $\mathcal{E}^*_{\lambda}.$
\label{l0}

\end{lem}
\begin{proof} To prove the lemma, we will use a method  similar to the one used in Lemma 1 in \cite{EH}. Let $C_R(0):= B_R(0) \times \R$, where $B_R(0) \subset \mathbb R^n$ denotes the open ball centered  at the origin and of radius $R>0.$  By Remark \ref{rem3}, there exists $u^R \in \E(C_R(0))$ such that
\begin{equation}
      \iint_{\mathbb{R}^{n+1}} u^R \, \overline{\phi} + \lambda^2 (S+D) \nabla_x  u^R \, \cdot \, \overline{\nabla_x  \phi} + \lambda^2 H_t \dhalf u^R\, \overline{\dhalf \phi}\, \d x \d t = \iint_{\mathbb{R}^{n+1}} f\, \overline{\phi}\, \d x \d t,
      \label{eqsolution}
\end{equation}
for all $\phi \in \E(C_R(0))$, and
\begin{align}
\label{ubound}
\|u^R\|_2+\|\lambda\nabla_x u^R \|_2+\|\lambda\dhalf u^R \|_2\lesssim \|f\|_{2}.
\end{align}
By Lemma \ref{space} (i), there exists a sequence $\{u^R_i\}$, $u^R_i \in C^{\infty}_0(C_R(0);\mathbb{C})$, such that
\begin{align}\label{bla}\lim_{i \to \infty} \iint_{\mathbb{R}^{n+1}} |u^R-u^R_i|^2 + |\nabla_x(u^R-u^R_i)|^2 + |\dhalf (u^R-u^R_i)|^2\, \d x \d t = 0.
\end{align}
Let  $\phi := e^{2\big(\frac{x \cdot \chi}{\lambda}\big)} \, u^R$, and $ \phi_i := e^{2\big(\frac{x \cdot \chi}{\lambda}\big)}  \, u^R_i$, for $i = 1, 2, \cdot \cdot \cdot$.  Using  Lemma \ref{space} (ii) and \eqref{bla}, we have $\phi \in \E(C_R(0))$ and
$$
\lim_{i \to \infty} \|\phi- \phi_i\|_{\E(C_R(0))}  =0.
$$
Hence,
\begin{align*}
\Re \iint_{\mathbb{R}^{n+1}} H_t \dhalf u^R\,  \overline{\dhalf (e^{2\big(\frac{x \cdot \chi}{\lambda} \big)}  u^R)} \, \d x \d t =& \lim_{i \to \infty} \Re \iint_{\mathbb{R}^{n+1}} H_t \dhalf u^R_i \,  \overline{\dhalf (e^{2\big(\frac{x \cdot \chi}{\lambda} \big)}  u^R_i )} \, \d x \d t.
\end{align*}
As $\dhalf H_t \dhalf =\partial_t$, integration by parts yields
\begin{align*}
 \Re \iint_{\mathbb{R}^{n+1}} H_t \dhalf u^R_i \,  \overline{\dhalf (e^{2\big(\frac{x \cdot \chi}{\lambda} \big)}  u^R_i )} \, \d x \d t =& \Re \iint_{\mathbb{R}^{n+1}} (\overline{u^R_i} \partial_t u^R_i)  e^{2\big(\frac{x \cdot \chi}{\lambda}\big)} \, \d x \d t \\ =& \frac{1}{2}  \iint_{\mathbb{R}^{n+1}}\partial_t  (|u^R_i|^2) \,  e^{2\big(\frac{x \cdot \chi}{\lambda} \big)} \, \d x \d t=0.
\end{align*}
Putting the conclusions in the last displays together, we can conclude that
\begin{align}\label{lyllo}
\Re \iint_{\mathbb{R}^{n+1}} H_t \dhalf u^R\,  \overline{\dhalf (e^{2\big(\frac{x \cdot \chi}{\lambda} \big)}  u^R)} \, \d x \d t =0.
\end{align}
Thus, setting $ \phi = e^{2\big(\frac{x \cdot \chi}{\lambda}\big)} \, u^R$ in \eqref{eqsolution}, together with \eqref{lyllo}, yield
\begin{align}\label{lyllo1-}
&\Re \iint_{\mathbb{R}^{n+1}} e^{2\big(\frac{x \cdot \chi}{\lambda} \big)} ( |u^R|^2  +\lambda^2 (S+D)\nabla_x  u^R \, \cdot   (\nabla_x\overline{u^R} +\frac {2\chi}\lambda \overline{u^R})  )\, \d x\d t\notag\\
&=\Re \iint_{\mathbb{R}^{n+1}} e^{2\big(\frac{x \cdot \chi}{\lambda}\big)} \, f\, \overline{u^R}\,\d x \d t.
\end{align}
By the anti-symmetry of $D$,
\begin{align*}
\Re \iint_{\mathbb{R}^{n+1}} e^{2\big(\frac{x \cdot \chi}{\lambda} \big)} \lambda^2 D\nabla_x  u^R \, \cdot   (\nabla_x\overline{u^R} +\frac {2\chi}\lambda \overline{u^R})\, \d x\d t&= \iint_{\mathbb{R}^{n+1}} \lambda D\nabla_x  |u^R|^2 \, \cdot {\chi} e^{2\big(\frac{x \cdot \chi}{\lambda}\big)} \, \d x\d t\\
&=\iint_{\mathbb{R}^{n+1}} \lambda D{\chi}\cdot\nabla_x  (|u^R|^2e^{2\big(\frac{x \cdot \chi}{\lambda}\big)} ) \, \d x\d t.
\end{align*}
Applying Lemma \ref{p3.1} in conjunction with \eqref{eqbmo},
\begin{align}\label{lyllo1}
 |\iint_{\mathbb{R}^{n+1}} \lambda D\nabla_x  (|u^R|^2e^{2\big(\frac{x \cdot \chi}{\lambda}\big)}) \, \cdot {\chi}  \, \d x\d t|\notag
 &\lesssim  \lambda |\iint_{\mathbb{R}} \| |u^R|e^{\big(\frac{x \cdot \chi}{\lambda}\big)} \chi \cdot \nabla_x  (|u^R|e^{\big(\frac{x \cdot \chi}{\lambda}\big)}) \,  \|_{\mathcal{H}^1(\R^n,\d x)}  \d t
 \\&\lesssim|\chi|\|\lambda\nabla_x  (u^Re^{\big(\frac{x \cdot \chi}{\lambda}\big)})\|_2 \|u^Re^{\big(\frac{x \cdot \chi}{\lambda}\big)}\|_2 \notag\\
&\lesssim|\chi|\|\lambda e^{\big(\frac{x \cdot \chi}{\lambda}\big)} \nabla_x u^R \|_2 \|e^{\big(\frac{x \cdot \chi}{\lambda}\big)} u^R \|_2\notag \\
&+|\chi|^2 \|e^{\big(\frac{x \cdot \chi}{\lambda}\big)} u^R \|_2^2.
\end{align}
Combining \eqref{lyllo1-}, \eqref{lyllo1}, and using  \eqref{ellip}, we deduce
\begin{align*}
&\|e^{\big(\frac{x \cdot \chi}{\lambda}\big)} u^R\|_2^2+\|\lambda e^{\big(\frac{x \cdot \chi}{\lambda}\big)} \nabla_x u^R \|_2^2\\&\lesssim |\chi|\|\lambda e^{\big(\frac{x \cdot \chi}{\lambda}\big)} \nabla_x u^R \|_2 \|e^{\big(\frac{x \cdot \chi}{\lambda}\big)} u^R \|_2
+|\chi|^2 \| e^{\big(\frac{x \cdot \chi}{\lambda}\big)} u^R\|_2^2+\| e^{2\big(\frac{x \cdot \chi}{\lambda}\big)} f u^R\|_1
\\ &\leq |\chi|\|\lambda e^{\big(\frac{x \cdot \chi}{\lambda}\big)} \nabla_x u^R \|_2 \|e^{\big(\frac{x \cdot \chi}{\lambda}\big)} u^R \|_2
+|\chi|^2 \| e^{\big(\frac{x \cdot \chi}{\lambda}\big)} u^R\|_2^2+\| e^{\big(\frac{x \cdot \chi}{\lambda}\big)}  f\|_2  \|e^{\big(\frac{x \cdot \chi}{\lambda}\big)} u^R\|_2,
\end{align*}
where we used Cauchy-Schwarz inequality on the last inequality. Hence, from the last display we see that if $|\chi|\leq\theta$ with $\theta$ small enough, which is independent of $R$, then
\begin{align}
\label{mainineq}
\| e^{\big(\frac{x \cdot \chi}{\lambda}\big)} u^R\|_2+\|\lambda e^{\big(\frac{x \cdot \chi}{\lambda}\big)} \nabla_x u^R \|_2&\lesssim \| e^{\big(\frac{x \cdot \chi}{\lambda}\big)} f\|_2,
\end{align}
and we emphasize that this estimate is uniform in $R$. We now want to pass to the limit $R\to\infty$ in \eqref{mainineq} to conclude the estimate with $u^R$ replaced by $\mathcal{E}_\lambda f$. To do this we first note that if $U\subset\mathbb R^n$ is a bounded domain such that $U\subset B_R(0)$, then \eqref{eqsolution} implies
\begin{equation}
      \iint_{\mathbb{R}^{n+1}} u^R \, \overline{\phi} + \lambda^2 (S+D) \nabla_x  u^R \, \cdot \, \overline{\nabla_x  \phi} - \lambda^2 u^R\, \overline{\partial_t\phi}\, \d x \d t = \iint_{\mathbb{R}^{n+1}} f\, \overline{\phi}\, \d x \d t,
      \label{eqsolutionaga}
\end{equation}
for all $\phi \in C_0^\infty(U\times\mathbb R)$. By \eqref{ubound}, or \eqref{mainineq}, and \eqref{eqsolutionaga} we have that
$\{u^R\}$ and $\{\partial_tu^R\}$ are uniformly bounded in $\L^2(I,H^1(U))$ and $\L^2(I,H^{-1}(U))$, respectively, for any compact time interval $I\subset\mathbb R$. Note that $H^1(U)$ is compactly embedded in $\L^2(U)$ by the Rellich-Kondrachov theorem, and that $\L^2(U)$ is continuously embedded in $H^{-1}(U)$. Therefore,  by the Aubin-Lions lemma, see \cite{lions1969}, we can conclude that $\{u^R\}$ has a strongly convergent subsequence in $\L^2(I,\L^2(U))$ for any compact interval $I$.  Consequently, using also the observation that $D^\ast\nabla_x\phi\in \L^2(\mathbb R^{n+1})$ whenever $\phi\in C_0^\infty(\mathbb R^{n+1})$, we can conclude based on \eqref{ubound} and by letting $R\to\infty$ that, up to a subsequence, $u^R$ converges to $\mathcal{E}_\lambda f$ strongly in $\L^2_{\loc}(\ree)$, and $\nabla_x u^R, \dhalf u^R$ converges to $\nabla_x \mathcal{E}_\lambda f, \dhalf \mathcal{E}_{\lambda} f$ weakly in $\L^2_{\loc}(\ree)$, respectively.  Hence, by Fatou's lemma and \eqref{mainineq}, we obtain (i) and (ii) of the lemma. Finally, (iii) can be proved by applying a similar reasoning to the equation
 \begin{align*}
      \iint_{\mathbb{R}^{n+1}} u\, \overline{\phi} + \lambda^2 (S+D) \nabla_x  u\, \cdot \, \overline{\nabla_x  \phi} + \lambda^2 H_t \dhalf u\, \overline{\dhalf \phi}\, \d x \d t = -\iint_{\mathbb{R}^{n+1}} \textbf{f}\, \cdot \overline{ \nabla_x  \phi}\, \d x \d t.
 \end{align*}
 We omit further details. The proofs in the case of $\mathcal{E}^*_{\lambda}$ are analogous.
\end{proof}

Given $E$, a compact subset of $\mathbb{R}^{n+1}$, we let
$$\pi(E) := \{t \in \R: (x,t) \in E \textup{ for some } x \in \R^n\}.$$
In other words, $\pi(E)$ is the orthogonal projection of $E$ onto the time-axis. Note that $\pi(E)$ is closed as $\pi(E)$ is compact and as the projection map is continuous.

\begin{lem}\label{lenewpiece}
Assume that $E,F$ are compact subsets of $\mathbb{R}^{n+1}$ and let $$d := \inf \{ |t-s|^{\frac{1}{2}}: t \in \pi(E), s \in \pi(F)\}.$$ Then, there exists a constant $c$, $1\leq c<\infty$, depending only on the structural constants, such that
\begin{align*}
\label{eq:off}
 \mathrm{(i)}& \quad \iint_{F} |\mathcal{E}_\lambda f|^2+|\lambda \nabla_x \mathcal{E}_\lambda f|^2\, \d x\d t \lesssim  e^{-\frac{d}{c\lambda}} \iint_{E} |f|^2\,\d x \d t,\notag\\
\mathrm{(ii)}& \quad \iint_{F}|\lambda\mathcal{E}_\lambda \div_x({\bf f})|^2 \, \d x \d t \lesssim e^{-\frac{d}{c\lambda}} \iint_{E} |{\bf f}|^2\, \d x \d t,
\end{align*}
whenever $f:\ree\to \IC$ and ${\bf f}:\ree\to \IC^{n}$  are supported in $E$, $f \in \L^2(\ree)$, and ${\bf f} \in \L^2(\ree;\mathbb{C}^n)$. The same statements are true with $\mathcal{E}_\lambda$ replaced by $\mathcal{E}_\lambda^\ast$.
\end{lem}
\begin{proof}  To prove this lemma we follow the approach outlined in \cite{AAEN} and \cite{N1}, both of which build on ideas developed in the proof of the original elliptic Kato problem in \cite{AHLMcT}. Using Lemma \ref{le8-},  we can assume, without loss of generality, that $\lambda \leq d$. Let $u :=\mathcal{E}_\lambda f$ and $\phi(t) := e^{\alpha \eta(t)} -1 $, where $\alpha = {\epsilon d}/{\lambda}$ and $\epsilon>0$ is small {enough} to be chosen later. Here $\eta \in C^{\infty}_0(\R)$ is a real-valued cut-off function satisfying $\eta = 1$ on ${\pi(F)}$, $\eta = 0$ on ${\pi(E)},$ and $|\partial_t \eta| \lesssim {1}/{d^2}.$  Then, $u \phi^2 \in \E(\ree)$ by Lemma \ref{space} (ii) and
\begin{align*}
      \iint_{\mathbb{R}^{n+1}} u \, \overline{u \phi^2} + \lambda^2 (S+D) \nabla_x  u \, \cdot \, \overline{\nabla_x  (u \phi^2)} - \lambda^2 \dhalf u\, \overline{H_t \dhalf (u \phi^2)}\, \d x \d t = \iint_{\mathbb{R}^{n+1}} f\, \overline{u \phi^2}\, \d x \d t=0,
\end{align*}
as the supports of  $f$ and $u \phi^2$ are disjoint. Since $\phi$ is real-valued and only depends on time,
$$
\Re D \, \nabla_x u \cdot \overline{\nabla_x (u \phi^2)} =\phi^2 \Re D\, \nabla_x u \cdot \overline{\nabla_x u} =0.
$$
Consequently, from the last two display we have
\begin{equation}\label{fy1}
   \iint_{\ree} |u|^2 \phi^2 + \lambda^2 \phi^2 \Re S\, \nabla_x  u \, \cdot \, \overline{\nabla_x  u} - \lambda^2 \Re \dhalf u\, \overline{H_t \dhalf (u \phi^2)}\, \d x \d t =0.
\end{equation}
Now, once the term involving $D$ is gone by the above deductions, the remaining argument follows along the lines of the corresponding arguments in  \cite{AAEN} and \cite{N1}. Since $u \phi^2 \in \E(\ree)$ and $u \in \E(\ree)$,  there exists a sequence $u_i \in C^{\infty}_0(\ree)$ such that $$\lim_{i \to \infty} \|(u_i-u) \phi^2 \|_{\E(\ree)} + \|u_i-u\|_{\E(\ree)}=0.$$ Hence,
\begin{align*}
\Re \iint_{\mathbb{R}^{n+1}} -\dhalf u\,  \overline{H_t \dhalf \bigl(  u \phi^2 \bigr)} \, \d x \d t= \lim_{i \to \infty} \Re \iint_{\mathbb{R}^{n+1}} -\dhalf u_i \,  \overline{H_t \dhalf \bigl( u_i \phi^2 \bigr)} \, \d x \d t.
\end{align*}
As $\dhalf H_t \dhalf =\partial_t$, and as $u_i \in C^{\infty}_0(\ree)$, $u_i\phi^2 \in C^{\infty}_0(\ree)$, integration by parts yields
\begin{align*}
\Re \iint_{\mathbb{R}^{n+1}} -\dhalf u_i \,  \overline{H_t \dhalf \bigl( u_i \phi^2 \bigr)} \, \d x \d t=& \Re \iint_{\mathbb{R}^{n+1}} (\overline{u_i} \partial_t u_i) \phi^2 \, \d x \d t\\=& \frac{1}{2}  \iint_{\mathbb{R}^{n+1}} \partial_t(|u_i|^2) \, \phi^2 \, \d x \d t \\=& \frac{1}{2}  \iint_{\mathbb{R}^{n+1}}  -|u_i|^2 \,  \partial_t  \phi^2 \, \d x \d t.
\end{align*}
Taking the limit $i \to \infty$ in the last display, we get
\begin{align}\label{fy2}
&\Re \iint_{\mathbb{R}^{n+1}} -\dhalf u\,  \overline{H_t \dhalf \bigl(  u \phi^2 \bigr)} \, \d x \d t=\iint_{\ree} -|u|^2 \phi \, \partial_t \phi \, \d x \d t.
\end{align}
Combining \eqref{fy1} and \eqref{fy2},
\begin{align}\label{fy3}
   \iint_{\ree} |u|^2 \phi^2 + \lambda^2 \phi^2 \Re S\, \nabla_x  u \, \cdot \, \overline{\nabla_x  u} - \lambda^2|u|^2 \phi \, \partial_t \phi\, \d x \d t =0.
\end{align}
Using \eqref{fy3} and the coercivity of $S$, see \eqref{ellip}, we deduce that
\begin{align*}
     \iint_{\ree} |u|^2 \phi^2 +  c_1\lambda^2  |\nabla_x u|^2 \phi^2 \, \d x \d t &\lesssim \lambda^4 \iint_{\ree} |u|^2 |\partial_t \phi|^2 \, \d x \d t\\
     &\lesssim \frac{\epsilon^2 \lambda^2 }{ d^2} \iint_{\ree} |u|^2 |\phi+1|^2 \, \d x \d t.
\end{align*}
Note that in the second inequality we have used the specific form of $\phi$ and the bound on $|\partial_t\eta|$.  Choosing $\epsilon$ small enough in the last display, and only depending on the structural constants,
we first conclude that
\begin{align*}
   & \iint_{\ree} |u|^2 \, \phi^2\, \d x \d t + c_1\lambda^2 \iint_{\ree} |\nabla_x u|^2 \, \phi^2\, \d x \d t\leq  \frac 1 4\iint_{\ree} |u|^2 |\phi+1|^2\, \d x \d t,
\end{align*}
and then
\begin{align*}
   & \iint_{\ree} |u|^2 \, \phi^2\, \d x \d t + \lambda^2 \iint_{\ree} |\nabla_x u|^2 \, \phi^2\, \d x \d t\lesssim \iint_{E} |f|^2\, \d x \d t,
\end{align*}
by an application of Lemma \ref{le8-} (i). Using that $\phi =e^{\alpha {\eta}}-1$, and once again Lemma \ref{le8-} (i), the inequality in the last display implies that
\begin{align*}
   & \iint_{\ree} |u|^2 \, e^{2\alpha \eta}\, \d x \d t + \lambda^2 \iint_{\ree}|\nabla_x u|^2 \, e^{2\alpha {\eta}}\, \d x \d t\lesssim \iint_{E} |f|^2\, \d x \d t.
\end{align*}
As $\eta=1$ on $F$,
\begin{align*}
   & e^{2\alpha} \iint_{F} |u|^2\, \d x \d t +  e^{2\alpha}\lambda^2 \iint_{F}|\nabla_x u|^2 \, \d x \d t\lesssim \iint_{E} |f|^2\, \d x \d t,
\end{align*}
proving the conclusion in $\mathrm{(i)}$. The proof of $\mathrm{(ii)}$ follows by duality, see Lemma 4.4 in \cite{AAEN} for example.
\end{proof}

\begin{lem} There exists a constant $c$, $1\leq c<\infty$, depending only on the structural constants, such that
$$
\begin{aligned}
\mathrm{(i)}& \quad \|\mathcal{E}_\lambda (f \, 1_{2^{k+1} \Delta \setminus 2^k \Delta})\|_{\L^2(\Delta)} \lesssim  e^{-\big(\frac{2^k \ell(\Delta)}{c\lambda}\big)}  \|f\|_{\L^2(2^{k+1} \Delta \setminus 2^k \Delta)}, \\
\mathrm{(ii)}& \quad  \|\lambda \nabla_x  \mathcal{E}_\lambda (f \, 1_{2^{k+1} \Delta \setminus 2^k \Delta})\|_{\L^2(\Delta)} \lesssim  e^{-\big(\frac{2^k \ell(\Delta)}{c\lambda}\big)}  \|f\|_{\L^2(2^{k+1} \Delta \setminus 2^k \Delta)}, \\
\mathrm{(iii)}& \quad \|\lambda \mathcal{E}_\lambda   \div_x (\textbf{f}\, \, 1_{2^{k+1} \Delta \setminus 2^k \Delta})\|_{\L^2(\Delta)}  \lesssim  e^{-\big(\frac{2^k \ell(\Delta)}{c\lambda}\big)}  \|\textbf{f}\,\|_{\L^2(2^{k+1} \Delta \setminus 2^k \Delta)},
\end{aligned}
$$
for all parabolic cubes $\Delta \subset \mathbb{R}^{n+1}$ and $f \in \L^2(\R^{n+1}),\, \textbf{f} \in \L^2(\R^{n+1};\mathbb{C}^{n}).$ The same statements are true with $\mathcal{E}_\lambda$ replaced by $\mathcal{E}_\lambda^\ast$.
\label{le8-+}
\end{lem}
\begin{proof} Our proof of this lemma is inspired by the proof of Lemma 2 in \cite{EH}. Let $\|x\|_{\infty} := \max \{|x_i|\}_{i=1}^n$ for $x=(x_1,\cdot \cdot \cdot,x_n) \in \mathbb{R}^n.$ As the class of operators considered is invariant under Euclidean translations we can in the following assume, without loss of generality, that $\Delta = Q \times I$ is centered at the origin and hence $$2^k \Delta = 2^k Q \times 4^k I = \Big\{(x,t) \in \mathbb{R}^{n+1}: 2\|x\|_{\infty}<  2^{k} \ell(\Delta)\mbox{ and }\sqrt{2|t|}< 2^{k} \ell(\Delta) \Big \},$$ for $k \in \mathbb{N}$. Let $A_i, B_i$, for $i = 1,...,n$, and  $\tilde A,\tilde B$, be defined according to
$$
\begin{aligned}
A_i :=&\Big \{(x,t) \in \mathbb{R}^{n+1}: \max(2\|x\|_{\infty},\sqrt{2|t|}) = 2x_i \Big\}, \\
B_i :=& \Big\{(x,t) \in \mathbb{R}^{n+1}: \max(2\|x\|_{\infty},\sqrt{2|t|}) = -2x_i \Big \},\\
\tilde A :=& \Big \{(x,t) \in \mathbb{R}^{n+1}: \max(2\|x\|_{\infty},\sqrt{2|t|}) = \sqrt{2t} \Big \}, \\
\tilde B :=& \Big \{(x,t) \in \mathbb{R}^{n+1}: \max(2\|x\|_{\infty},\sqrt{2|t|}) = \sqrt{-2t} \Big \}.
\end{aligned}
$$
Then,
$$
\begin{aligned}
\mathcal{E}_\lambda \big(f\, 1_{2^{k+1}\Delta \setminus 2^k \Delta}\big) &= \sum_{i=1}^n  \mathcal{E}_\lambda \big(f\, 1_{A_i \cap (2^{k+1}\Delta \setminus 2^k \Delta)}\big)\, +\,  \mathcal{E}_\lambda \big(f\, 1_{B_i \cap (2^{k+1}\Delta \setminus 2^k \Delta)}\big)\\&
+  \mathcal{E}_\lambda \big(f\, 1_{\tilde A \cap (2^{k+1}\Delta \setminus 2^k \Delta)}\big)\, +\,  \mathcal{E}_\lambda \big(f\, 1_{\tilde B \cap (2^{k+1}\Delta \setminus 2^k \Delta)}\big).
\end{aligned}
$$
First, we now apply Lemma \ref{l0} to the function $f \, 1_{A_i \cap (2^{k+1}\Delta \setminus 2^k \Delta)}$ with $\chi = -\theta \,e_i$ where $e_i$ is the unit vector in the $i$-th coordinate direction. In this case
$$e^{-\big( \frac{\theta \ell(\Delta)}{\lambda}\big)} \leq   e^{\big(\frac{x \cdot \chi }{\lambda} \big)},$$
for $(x,t) \in \Delta$, and
$$ e^{\big(\frac{x \cdot \chi }{\lambda}\big)} \leq e^{-\big( \frac{\theta\, 2^k \ell(\Delta)}{2\lambda}\big)},$$
for $(x,t) \in A_i \cap (2^{k+1}\Delta \setminus 2^k \Delta)$. Hence, by Lemma \ref{l0}, we obtain
$$
\| \mathcal{E}_\lambda \big(f\, 1_{A_{i} \cap (2^{k+1}\Delta \setminus 2^k \Delta)}\big)\|_{\L^2(\Delta)} \lesssim e^{-\big(\frac{2^k \ell(\Delta)}{c\lambda }\big)} \|f \, 1_{A_{i}}\|_{\L^2(2^{k+1} \Delta \setminus 2^k \Delta)}.
$$
Second, we apply Lemma \ref{lenewpiece} to $f\, 1_{\tilde A \cap (2^{k+1}\Delta \setminus 2^k \Delta)}$ and the sets $E =  \overline{\tilde A \cap (2^{k+1} \Delta \setminus 2^k \Delta)}$, $F = \overline{\tilde A \cap \Delta}$. Hence,
$$
\| \mathcal{E}_\lambda \big(f\, 1_{\tilde A \cap (2^{k+1}\Delta \setminus 2^k \Delta)}\big)\|_{\L^2(\Delta)} \lesssim e^{-\big(\frac{2^k \ell(\Delta)}{c\lambda}\big)} \|f \, 1_{\tilde A}\|_{\L^2(2^{k+1} \Delta \setminus 2^k \Delta)}.
$$
The terms involving $f\, 1_{B_i \cap (2^{k+1}\Delta \setminus 2^k \Delta)}$ and $f\, 1_{\tilde B \cap (2^{k+1}\Delta \setminus 2^k \Delta)}$ can be treated similarly. (ii) and (iii) of the lemma, as well as the estimates for $\mathcal{E}^*_{\lambda}$, can be handled using the same ideas.
\end{proof}
\begin{cor}
There exists a constant $c$, $1\leq c<\infty$, depending only on the structural constants, such that
$$
\begin{aligned}
\mathrm{(i)}& \quad \|\mathcal{E}_\lambda (f \, 1_{ \Delta})\|_{\L^2(2^{k+1} \Delta \setminus 2^k \Delta)} \lesssim  e^{-\big(\frac{2^k \ell(\Delta)}{c\lambda}\big)}  \|f\|_{\L^2( \Delta)}, \\
\mathrm{(ii)}& \quad  \|\lambda \nabla_x  \mathcal{E}_\lambda (f \, 1_{\Delta})\|_{\L^2(2^{k+1} \Delta \setminus 2^k \Delta)} \lesssim  e^{-\big(\frac{2^k \ell(\Delta)}{c\lambda}\big)}  \|f\|_{\L^2( \Delta)}, \\
\mathrm{(iii)}& \quad \|\lambda \mathcal{E}_\lambda   \div_x (\textbf{f}\, \, 1_{ \Delta})\|_{\L^2(2^{k+1} \Delta \setminus 2^k \Delta)}  \lesssim  e^{-\big(\frac{2^k \ell(\Delta)}{c\lambda}\big)}  \|\textbf{f}\,\|_{\L^2( \Delta)},
\end{aligned}
$$
for all parabolic cubes $\Delta \subset \mathbb{R}^{n+1}$ and $f \in \L^2(\R^{n+1}),\, \textbf{f} \in \L^2(\R^{n+1};\mathbb{C}^{n}).$ The same statements are true with $\mathcal{E}_\lambda$ replaced by $\mathcal{E}_\lambda^\ast$.
\label{corof}
\end{cor}
 \begin{proof}
 The estimates follows by a duality argument and Lemma \ref{le8-+}. To prove (i),
 \begin{align*}
   \|\mathcal{E}_\lambda (f \, 1_{ \Delta})\|_{\L^2(2^{k+1} \Delta \setminus 2^k \Delta)} & =  \sup_{g } \iint_{\ree} \mathcal{E}_\lambda (f \, 1_{ \Delta}) \, \cl{g} \, \d x \d t  \\ &= \sup_{g}  \iint_{\ree}  f \, 1_{ \Delta} \, \cl{\mathcal{E}^*_\lambda g} \, \d x \d t  \\ & \lesssim \sup_ {g} \| f \|_{\L^2(\Delta)}  \|\mathcal{E}^*_\lambda g\|_{\L^2(\Delta)} \lesssim e^{-\big(\frac{2^k \ell(\Delta)}{c\lambda}\big)}   \| f\|_{\L^2(\Delta)} ,
 \end{align*}
 where supremum is taken over all $g \in C^{\infty}_0(2^{k+1} \Delta \setminus 2^k \Delta;\IC)$, satisfying $\|g\|_2 = 1$, and Lemma \ref{le8-+} is used on the last inequality.  The proofs of the other estimates are analogous. We omit further details.\end{proof}

\section{Additional estimates related to the principal part approximation}\label{additional}  We here deduce a number of estimates  for the operator
\begin{align}\label{opa0}\mathcal{U}_\lambda :=\lambda\mathcal{E}_\lambda \div_x,
 \end{align}
and for the operator
\begin{align}\label{opa1}
\mathcal{R}_\lambda &:=\mathcal{U}_\lambda A -(\mathcal{U}_\lambda A)\mathcal{A}_\lambda.
 \end{align}
Recall that the dyadic averaging operator $\mathcal{A}_\lambda $ was introduced in $\eqref{dy}$ and note that we must address the very definitions of $(\mathcal{U}_\lambda A)$ and
$\mathcal{R}_\lambda$, where the first one acts as the multiplication with $\mathcal{U}_\lambda A$. In Lemma \ref{lemedda} we will justify that $\mathcal{U}_\lambda A \in \L^{2}_{\loc}(\R^{n+1})$. The estimates we will derive in this section will be used in the proof of
 Theorem \ref{thm:Kato}, see Section \ref{sec2}.

 Using Lemma $\ref{le8-}$ and Lemma $\ref{le8-+}$ we have
\begin{equation}
\begin{aligned}
\sup_{\lambda>0} \|\, \mathcal{U}_\lambda \textbf{f}\,\|_{2}\lesssim \|\textbf{f}\,\|_{2},
\end{aligned}
\label{stand-}
\end{equation}
 for all  ${\bf f} \in \L^2(\ree;\IC^n)$ and
\begin{equation}
\begin{aligned}
\|\mathcal{U}_\lambda(\,\textbf{f}\,\, 1_{2^{k+1} \Delta \setminus 2^k \Delta})\|_{\L^2(\Delta)}^2 \lesssim e^{-\big(\frac{2^k \ell(\Delta) }{c \lambda}\big)} \|\,\textbf{f}\,\|^2_{\L^2(2^{k+1} \Delta \setminus 2^k \Delta)},
\end{aligned}
\label{stand}
\end{equation}
 for all  ${\bf f} \in \L^2_{\loc}(\ree;\IC^n)$,  for all integers $k\geq 1$, and for a constant $c$ depending only on structural constants. In \eqref{stand} $\Delta$ is, as usual,  a parabolic cube.

\begin{rem}
\label{keyrem} Let $\Delta_0=Q_0\times I_0$ where $Q_0 \subset \R^n$ is a cube as in Subsection \ref{keysec}. Let  
 \begin{align}
   \label{apa}
   \mbox{$\textbf{f}\, = \barint_{Q_0} A_i$ for some $i$, $1 \leq i \leq n$, {where  $A = (A_1, \cdot \cdot \cdot,A_n)$},}
    \end{align}
and let $\chi_k(x,t) = \chi\big( 2^{-k}{x}, 2^{-2k}{t}\big)$ for $x \in \R^n,t \in \R$, where $\chi \in C^{\infty}_0(2\Delta_0)$ is a cut-off function satisfying $\chi = 1$ on $\Delta_0$.  Then, we define
 \begin{align}
     \mathcal{U}_\lambda(\, \textbf{f}\,) := \lim_{k \to \infty}\mathcal{U}_\lambda( \chi_k \, \textbf{f}\,),
 \label{newdef}
 \end{align}
 which exists as an element in $\L^2_{\loc}(\ree)$.  Indeed, using  the fact that $$\nabla_x \chi_{k+l} =0, \quad \textup{on } \ree \setminus (2^{k+l+1} \Delta_0 \setminus 2^{k+l} \Delta_0),$$ for every $l \geq 1$ and Lemma \ref{le8-+}, we obtain
 \begin{align*}
    \|\mathcal{U}_\lambda (\, \chi_{k+l} \,  \textbf{f}\,) \|_{\L^2(2^l \Delta_0)}  & =\bigl \|\mathcal{U}_\lambda \bigl( \chi_{k+l}  \,\barint_{Q_0} A_i\,\bigr)\bigr \|_{\L^2(2^l \Delta_0)}\\
     &=\bigl \| \lambda\mathcal{E}_{\lambda} \bigl( \nabla_x \chi_{k+l} \cdot \barint_{Q_0} \, A_i   \bigr) \bigr \|_{\L^2(2^l\Delta_0)} \\ & \lesssim \lambda e^{-\big(\frac{2^{k+l} \ell(\Delta_0 ) }{c \lambda}\big)}  \bigl \|   \nabla_x \chi_{k+l} \cdot  \barint_{Q_0} \, A_i  \bigr \|_{\L^2(2^{k+l+1} \Delta_0 \setminus 2^{k+l} \Delta_0)}\\
       &\lesssim \lambda |\Delta_0|^{1/2}e^{-\big(\frac{2^{k+l} \ell(\Delta_0 ) }{c \lambda}\big)}  2^{\big(\frac{n(k+l)}{2}\big)},
 \end{align*}
for every $l,k \geq 1$.  Letting $k \to \infty$ in the above inequality, we see that
    \begin{align}
   \label{apa+}\mbox{$\mathcal{U}_\lambda(\, \barint_{Q_0} A_i\,) = 0$ in the sense of $\L^2_{\loc}(\ree;\mathbb{C}^n)$.}
    \end{align}
 The fact that the definition of $\mathcal{U}_\lambda(\, \textbf{f}\,)$ is independent of the particular choice of $\chi$ is a consequence of the off-diagonal estimate for $\mathcal{U}_\lambda$ stated in Lemma \ref{le8-+} (iii). Note that if $\textbf{f} \in \L^2(\ree;\mathbb{C}^n)$, then the definition of $\mathcal{U}_\lambda(\, \textbf{f}\,)$ in \eqref{newdef} coincides with our previous definition of $ \mathcal{U}_\lambda(\, \textbf{f}\,)$ {by \eqref{stand-}}.
 \end{rem}

\begin{lem}\label{lemedda} Consider  $A = (A_1, \cdot \cdot \cdot,A_n)$.  Then
$$
\sup_{\lambda>0} \biggl( \sup_{\Delta} \frac{1}{|\Delta|} \|\mathcal{U}_{\lambda}\, A_i \,\|^2_{\L^2(\Delta)}\biggr) \lesssim 1,
  $$ where the second supremum is taken, for fixed $\lambda$,  over all the parabolic cubes $\Delta \subset \mathbb{R}^{n+1}$ of size $\lambda.$ In particular,
  $\mathcal{U}_{\lambda} A_i \in \L^2_{\loc}(\ree)$ for all $1 \leq i \leq n$.
\end{lem}
\begin{proof}  By Remark \ref{keyrem}, we have $ \mathcal{U}_{\lambda}(\barint_{Q} A_i) = 0$ for every cube $Q \subset \R^n.$ Furthermore, for fixed $i$,  $1 \leq i \leq n$, and a parabolic cube $\Delta = Q \times I$, by \eqref{bmoo1} we have
\begin{align*}
    \frac{1}{|\Delta|}  \bigl\|\, A_i\,-\barint_{Q} \, A_i \,  \bigr\|^2_{\L^2(\Delta)} =  \barint_{I}  \barint_{Q} \bigl|\, A_i\,-\barint_{Q}\, A_i \,  \bigr|^2 {\, \d x \d t} \lesssim 1.
\end{align*}
     To start the main argument, fix a parabolic cube $\Delta$ of size $\lambda$. As the class of operators and coefficients considered is invariant under Euclidean translations we can in the following assume, without loss of generality,  that $\Delta=Q \times (-\lambda^2/2,\lambda^2/2)$. Then, $2^k \Delta = 2^k Q \times (-(2^k \lambda)^2/2,(2^k \lambda)^2/2)$ for $k \geq 0$. Using that  $\mathcal{U}_{\lambda}( \barint_{Q} \, A_i \, )=0$ we have
$$\begin{aligned}
 {\|\mathcal{U}_{\lambda} \, A_i \,\|}_{\L^2(\Delta)} &=  \bigl\|\mathcal{U}_{\lambda} \bigl( \, A_i \,- \barint_Q \, A_i \, \bigr)\bigr\|_{\L^2(\Delta)}\\ &= \bigl\|\mathcal{U}_{\lambda} \bigl((\, A_i \,-\barint_Q \, A_i \, )1_{\Delta} +  \sum_{k=0}^{\infty} (\, A_i \,- \barint_Q \, A_i \, ) 1_{2^{k+1}\Delta \setminus 2^k \Delta} \bigr)\bigr\|_{\L^2(\Delta)}\\
&\leq  \bigl\|\mathcal{U}_{\lambda} \bigl( (\, A_i \,-\barint_Q \, A_i \,)1_{\Delta} \bigr)\bigr\|_{\L^2(\Delta)} + \sum_{k=0}^{\infty} \bigl\|\mathcal{U}_{\lambda} \bigl(  (\, A_i \,- \barint_Q\, A_i \,) 1_{2^{k+1}\Delta \setminus 2^k \Delta} \bigr)\bigr\|_{\L^2(\Delta)},\end{aligned}
$$
where the Minkowski inequality is used in the last inequality. Hence, by (\ref{stand-}) and (\ref{stand}) we obtain
$$\begin{aligned}
\frac{1}{|\Delta|^{\frac{1}{2}}} \|\mathcal{U}_{\lambda} \, A_i \,\|_{\L^2(\Delta)} &\lesssim \frac{1}{|\Delta|^{\frac{1}{2}}} \bigl\| \, A_i \,-\barint_Q \, A_i \, \bigr\|_{\L^2(\Delta)} + \sum_{k=0}^{\infty}   e^{-\big(\frac{2^k}{c}\big)}    \frac{1}{|\Delta|^{\frac{1}{2}}} \bigl\| \, A_i \,-\barint_Q \, A_i \, \bigr\|_{\L^2(2^{k+1} \Delta \setminus 2^{k} \Delta)}.
\end{aligned}
$$
Introducing a telescoping sum we have
\begin{align}\label{edda1}
\bigl\| \, A_i \,-\barint_Q \, A_i \, \bigr\|_{\L^2(2^{k+1} \Delta \setminus 2^{k} \Delta)}& \leq \bigl\| \bigl(\, A_i \,-\barint_{2^{k+1} Q} \, A_i \,  \bigr)+ \sum_{l=1}^{k+1} \bigl(\barint_{2^{l} Q} \, A_i \, - \barint_{2^{l-1} Q} \, A_i \,  \bigr)\bigr\|_{\L^2(2^{k+1} \Delta)}\notag\\& \leq \bigl\| \, A_i \,-\barint_{2^{k+1} Q} \, A_i \, \bigl | \bigl|_{\L^2(2^{k+1} \Delta)}+ \sum_{l=1}^{k+1}  \bigl| \bigl|  \barint_{2^{l} Q} \, A_i \,-\barint_{2^{l-1} Q} \, A_i \, \bigr\|_{\L^2(2^{k+1} \Delta)},
\end{align}
where again the Minkowski inequality is used in the last inequality. Also,
$$\begin{aligned}
\bigl| \bigl|  \barint_{2^{l} Q}\, A_i \, - \barint_{2^{l-1} Q} \, A_i \, \|_{\L^2(2^{k+1} Q)}  & \leq |2^{k+1} Q|^{\frac{1}{2}} \bigl| \barint_{2^{l{-1}} Q} \, \bigl(A_i \,-\barint_{2^{{l}} Q} \, A_i\bigr )\d x \, \bigr| \\ & \lesssim  |2^{k+1} Q|^{\frac{1}{2}}  \barint_{2^{l} Q} \bigl|\, A_i \,-\barint_{2^l Q} \, A_i \,  \bigr|\d x \\ & \lesssim |2^{k} Q|^{\frac{1}{2}}  \bigl(\barint_{2^l Q}\bigl| \, A_i \,- \barint_{2^l Q} \, A_i \, \bigr|^2\d x\bigr)^{\frac{1}{2}},
\end{aligned}
$$
for $1\leq l\leq k+1$, where the Cauchy-Schwarz inequality is used in the third inequality. Using the deduction in the last display and \eqref{bmoo1} we conclude that
\begin{align}\label{edda2}
\bigl| \bigl|  \barint_{2^{l} Q} \, A_i \,-\barint_{2^{l-1} Q} \, A_i \, \bigr\|_{\L^2(2^{k+1} \Delta)}&\lesssim\bigl (|2^{k} Q|2^{2k}\lambda^2 \bigr )^{1/2}\lesssim 2^{k/2} |\Delta|^{1/2}.
\end{align}
Similarly we see that
\begin{align}\label{edda3}
\bigl\| \, A_i \,-\barint_{2^{k+1} Q} \, A_i \, \bigl | \bigl|_{\L^2(2^{k+1} \Delta)}&\lesssim 2^{k/2} |\Delta|^{1/2}.
\end{align}
Combining the estimates in \eqref{edda1}-\eqref{edda3} we conclude that
\begin{align}\label{edda4}
\bigl\| \, A_i \,-\barint_Q \, A_i \, \bigr\|_{\L^2(2^{k+1} \Delta \setminus 2^{k} \Delta)}& \lesssim k2^{k/2} |\Delta|^{1/2}.
\end{align}
 Collecting the estimates derived we deduce that
 $$\begin{aligned}
\frac{1}{|\Delta|^{\frac{1}{2}}} \|\mathcal{U}_{\lambda} \, A_i \,\|_{\L^2(\Delta)} &\lesssim \frac{1}{|\Delta|^{\frac{1}{2}}} \bigl\| \, A_i \,-\barint_Q \, A_i \, \bigr\|_{\L^2(\Delta)} + \sum_{k=0}^{\infty}   e^{-\big(\frac{2^k}{c}\big)}    \frac{1}{|\Delta|^{\frac{1}{2}}} \bigl\| \, A_i \,-\barint_Q \, A_i \, \bigr\|_{\L^2(2^{k+1} \Delta \setminus 2^{k} \Delta)}\\
&\lesssim \frac{1}{|\Delta|^{\frac{1}{2}}} \bigl\| \, A_i \,-\barint_Q \, A_i \, \bigr\|_{\L^2(\Delta)}+ \sum_{k=0}^{\infty}  k2^{k/2} e^{-\big(\frac{2^k}{c}\big)}\lesssim 1
\end{aligned}
$$
by yet an other application of \eqref{bmoo1}. In particular, we can conclude $\mathcal{U}_{\lambda} A_i $ is well-defined  and belongs to $\L^2_{\loc}(\ree)$. \end{proof}

\begin{lem} Consider  $A = (A_1, \cdot \cdot \cdot,A_n)$. Then,  $\mathcal{U}_{\lambda} A_i \in \L^2_{\loc}(\ree;\IC^n)$ for all $1 \leq i \leq n$ by Lemma \ref{lemedda}, and
\begin{align}\label{eka2}
 \sup_{f} \biggl( \sup_{\lambda>0} \|(\mathcal{U}_{\lambda}\, A_i \,) \mathcal{A}_\lambda  f\|_{2} \biggr) \lesssim 1,
\end{align}
where the first supremum is taken over all $f \in \L^2(\ree)$ with $\|f\|_2 = 1$.
\label{le11+}
\end{lem}

\begin{proof} Fix $\lambda>0$ and recall the dyadic averaging operator $\mathcal{A}_\lambda $ introduced in $\eqref{dy}$. For $\lambda$ fixed, let $\{\Delta\}$  dyadic decomposition of $\ree$ by parabolic cubes such that $\ell(\Delta) = \lambda$ for very $\Delta$ in the collection. Using that $\mathcal{A}_\lambda  f$ is constant on  $\Delta$ in the collection we have
\begin{align*}
    \iint_{\Delta} |(\mathcal{U}_{\lambda} \, A_i\,) \mathcal{A}_\lambda  {f}|^2 \, \d x \d t = \iint_{\Delta} |\mathcal{U}_{\lambda} \, A_i\,|^2\, \d x \d t \bariint_{\Delta} |\mathcal{A}_\lambda  {f}|^2\, \d x \d t \lesssim  \iint_{\Delta}   |\mathcal{A}_\lambda  {f}|^2\, \d x \d t.
\end{align*}
where we in the last inequality have used Lemma \ref{lemedda}.
 In conclusion,
\begin{align*}
    \|(\mathcal{U}_{\lambda} \, A_i\,) \mathcal{A}_\lambda  {f}\|^2_{2} \lesssim \sum_{\Delta}  \iint_{\Delta}  |\mathcal{A}_\lambda  {f}|^2\, \d x \d t\lesssim \|\mathcal{A}_\lambda  {f}\|^2_{2} \lesssim\| {f}\|^2_{2},
\end{align*}
where we have used $\mathcal{A}_\lambda  {f} \lesssim \mathcal{M}^{(1)}(\mathcal{M}^{(2)}(f))$ in the last inequality. \end{proof}

Next, consider the operator \begin{align*}
\mathcal{R}_\lambda &:=\mathcal{U}_\lambda A -(\mathcal{U}_\lambda A)\mathcal{A}_\lambda
 \end{align*}
 introduced in \eqref{opa1}.  In  the proof of
 Theorem \ref{thm:Kato}, see Section \ref{sec2}, we will need that $\mathcal{R}_\lambda (\nabla_x f)$ is well-defined, and that
 \begin{align}
 \label{newes}
     \|\mathcal{R}_{\lambda} (\nabla_x f)\|_{2} \lesssim \|\lambda \nabla_x (\nabla_x f)\|_{2} + \|\lambda^2 \partial_t (\nabla_x f)\|_{2},
 \end{align}
 for every $ f \in C_0^{\infty}(\ree)$. Note that as $A$ is not bounded, we  cannot deduce \eqref{newes} directly from previous works, like for instance \cite[Lem. 2.27]{N1}. What save us are the reverse H{\"o}lder inequalities stated in \eqref{eq:inverseholdergradient} in Remark \ref{conseq}, and that the fact that in
 \eqref{newes} the operator $\mathcal{R}_\lambda $ is applied to $\nabla_x f$.

 Furthermore, recall that given a cube $Q_0\subset\mathbb R^n$ we have
\begin{align}\label{formha}
    \div_x A \nabla_x = \div_x \bigl(A-\barint_{Q_0} D\bigr) \nabla_x,
\end{align}
in the weak sense, see Subsection 2.8. Consequently, in our arguments we can always, without loss of generality and whenever we prefer,  replace $D$ by $D - \barint_{Q_0} D$, and hence work under the assumption that
 \begin{align}\label{normal+a}
 \bigl(\barint_{Q_0} D\bigr )(t)= 0\mbox{ for  ${t \in \R}$ and }D \in \L^p_{\loc}(\ree)\mbox{ for all }p \geq 1.
 \end{align}
 We have used this reduction at instances. However, in the following we will, to help the reader to exactly see how the conditions on the coefficients enter into the argument, at instances work with the representation of $\div_x A \nabla_x$ in \eqref{formha}. Accepting this we write
\begin{align}\label{form1-}
    \mathcal{U}_\lambda A \nabla_x f =  \mathcal{U}_\lambda \bigl(A-\barint_{Q_0} D\bigr) \nabla_x f,
\end{align}
and
\begin{align}\label{form1}
    &\mathcal{R}_\lambda(\nabla_x f) =\mathcal{U}_\lambda A (\nabla_x f) -(\mathcal{U}_\lambda A)\mathcal{A}_\lambda(\nabla_x f)=\mathcal{U}_\lambda \bigl(A-\barint_{Q_0} D\bigr) (\nabla_x f) -(\mathcal{U}_\lambda A)\mathcal{A}_\lambda(\nabla_x f).
\end{align}
We prove the following lemma.

 \begin{lem}\label{newes+}
      Let $\mathcal{R}_\lambda$ be defined as \eqref{opa1}. Then, \eqref{newes} holds for all $ f \in C_0^{\infty}(\ree)$.
  \end{lem}
\begin{proof}  Note that  $\mathcal{R}_\lambda 1$ is well-defined by Lemma \ref{le11+}, and that $\mathcal{R}_\lambda 1 = 0.$  To start the argument we first intend to prove that $\mathcal{R}_\lambda$ satisfies the estimates
\begin{align}\label{keyedd}
\mathrm{(i)}& \quad \|\mathcal{R}_{\lambda} (\nabla_x f)\|_{2} \lesssim \|\nabla_x f\|_2, \notag\\
  \mathrm{(ii)}& \quad\|\mathcal{R}_{\lambda}((\nabla_x f)\, 1_{2^{k+1} \Delta \setminus 2^k \Delta})\|_{\L^2(\Delta)}^2 \lesssim  e^{-\big(\frac{2^k \ell(\Delta)}{c\lambda}\big)} \|\nabla_x f\|^2_{\L^2(2^{k+1} \Delta \setminus 2^k \Delta)},
\end{align}
for every $f \in  C_0^{\infty}(\ree)$, with a constant $c>0$ depending only on the structural constants. Here, $\Delta=Q\times I$ is a parabolic cube of size $\lambda\sim \ell(\Delta)$, and $k$ is a positive integer satisfying $k\geq 11$. We prove  $\mathrm{(i)}$ and $\mathrm{(ii)}$ by a duality argument.  By Lemma \ref{lemedda} and Lemma \ref{le11+}, we have that $(\mathcal{U}_\lambda A)\mathcal{A}_\lambda \nabla_x f$ is well-defined and
\begin{align*}
    \|(\mathcal{U}_\lambda A) \mathcal{A}_\lambda \nabla_x f\|_2 \lesssim \|\nabla_x f\|_2.
\end{align*}
Also,
\begin{align*}
{\|\mathcal{U}_\lambda A (\nabla_x f)\|_2 =} \sup_g   \iint_{\ree} (\mathcal{U}_\lambda A (\nabla_x f)) \, \cl{g} \, \d x \d t & = - \sup_g \iint_{\ree} A (\nabla_x  f) \cdot \cl{ \lambda \nabla_x \mathcal{E}^*_{\lambda}  g} \, \d x \d t \\ & \lesssim \sup_g \|\nabla_x  f\|_2 \, \|\lambda \nabla_x \mathcal{E}^*_{\lambda}  g \|_2 \lesssim   \|\nabla_x  f\|_2,
\end{align*}
 where the supremum is taken with respect to all $g \in \L^2(\ree)$, satisfying $\|g\|_2 = 1$. In this deduction, Lemma \ref{p3.1} in conjunction with \eqref{eqbmo} is used on the first inequality, and Lemma \ref{le8-} {(i)} is used on the last inequality. Hence, the proof of $\mathrm{(i)}$ is completed.  For the second part, by $\mathrm{(i)}$, it is sufficient to prove $\mathrm{(ii)}$ for $k$ large enough and as we will see $k\geq 11$ is sufficient. Note that
 \begin{align*}
      \| (\mathcal{U}_\lambda A) \mathcal{A}_\lambda( (\nabla_x f) 1_{2^{k+1} \Delta \setminus 2^k \Delta})\|_{\L^2(\Delta)}= 0,
 \end{align*}
 for $k\geq 11$. Hence we are left with the task of estimating, for $k\geq 11$,
 \begin{align*}
 \|\mathcal{U}_\lambda A ((\nabla_x f)1_{2^{k+1} \Delta \setminus 2^k \Delta} )\|_{\L^2(\Delta)} &  = \sup_g   \iint_{\ree} (\mathcal{U}_\lambda A ((\nabla_x f)1_{2^{k+1} \Delta \setminus 2^k \Delta} )) \, \cl{g}\,\d x\d t
 \end{align*}
 where the supremum is taken with respect to all $g \in C^{\infty}_0(\Delta;\IC)$ such that $\|g\|_2 = 1$. Using \eqref{form1-} we have
 \begin{align*}
  \iint_{\ree} (\mathcal{U}_\lambda A ((\nabla_x f)1_{2^{k+1} \Delta \setminus 2^k \Delta} )) \, \cl{g}\,\d x\d t= \iint_{\ree} (\mathcal{U}_\lambda\bigl (A-\barint_{Q} D\bigr ) ((\nabla_x f)1_{2^{k+1} \Delta \setminus 2^k \Delta} )) \, \cl{g}\,\d x\d t
 \end{align*}
 Consequently,
 \begin{align}\label{aabb}
 \|\mathcal{U}_\lambda A ((\nabla_x f)1_{2^{k+1} \Delta \setminus 2^k \Delta} )\|_{\L^2(\Delta)} &  = \sup_g   \iint_{\ree} (\mathcal{U}_\lambda
 \bigl (A-\barint_{Q} D\bigr ) ((\nabla_x f)1_{2^{k+1} \Delta \setminus 2^k \Delta} )) \, \cl{g}\,\d x\d t\notag\\ & =  - \sup_g \iint_{\ree} \bigl (A-\barint_{Q} D\bigr ) ((\nabla_xf)1_{2^{k+1} \Delta \setminus 2^k \Delta})) \cdot \cl{ \lambda \nabla_x \mathcal{E}^*_{\lambda}  g} \, \d x \d t \notag\\ &=- \sup_g \iint_{\ree}  ((\nabla_xf)1_{2^{k+1} \Delta \setminus 2^k \Delta})) \cdot \bigl (\bigl (A-\barint_{Q} D\bigr )^\ast \cl{ \lambda \nabla_x \mathcal{E}^*_{\lambda}  g}\bigr) \, \d x \d t\notag\\
 &\lesssim \sup_g \|(\nabla_xf)1_{2^{k+1} \Delta \setminus 2^k \Delta}\|_2 \, \bigl \| \lambda \bigl(A-\barint_{Q} D\bigr)^\ast \nabla_x \mathcal{E}^*_{\lambda}  g\bigr\|_{\L^2( 2^{k+1} \Delta \setminus 2^{k} \Delta)},
 \end{align}
 where $\bigl (A-\barint_{Q} D\bigr )^\ast$ is the (Hermitian) transpose of $\bigl (A-\barint_{Q} D\bigr )$. Let $p>2$ be as in Theorem \ref{thm:reverseholdergradient}. Using H{\"o}lder's inequality,
 \begin{align*}
  \bigl \| \lambda \bigl(A-\barint_{Q} D\bigr)^\ast \nabla_x \mathcal{E}^*_{\lambda}  g\bigr\|_{\L^2( 2^{k+1} \Delta \setminus 2^{k} \Delta)}&\leq
  \bigl \| \bigl(A-\barint_{Q} D\bigr)^\ast\bigr\|_{\L^{2p/(p-2)} (2^{k+1} \Delta)} \bigl \|\lambda\nabla_x \mathcal{E}^*_{\lambda}  g\bigr\|_{\L^p( 2^{k+1} \Delta \setminus 2^{k} \Delta)}.
 \end{align*}
 By the John-Nirenberg’s inequality, see \eqref{magicalineq},
 \begin{align*}
  \bigl \| \bigl(A-\barint_{Q} D\bigr)^\ast\bigr\|_{\L^{2p/(p-2)} (2^{k+1} \Delta)} \lesssim k^{(p-2)/(2p)}|2^k\Delta|^{(p-2)/(2p)}.
 \end{align*}
As $g \in C^{\infty}_0(\Delta;\IC)$ , we have, using \eqref{eq:inverseholdergradient} stated in Remark \eqref{conseq},
\begin{equation*}
    \begin{aligned}
    & \| \nabla_x \mathcal{E}^*_{\lambda} g\|_{\L^p(2^{k+1} \Delta \setminus 2^k \Delta)} \lesssim |2^k\Delta|^{(2-p)/(2p)}\|\nabla_x \mathcal{E}^*_{\lambda} g\|_{\L^2(2^{k+5} \Delta \setminus 2^{k-4} \Delta)},\\
    \end{aligned}
    \end{equation*}
    for all $k\geq 11$ and for $p>2$ as in Theorem \ref{thm:reverseholdergradient}. Combining the estimates in the last three displays,
    \begin{align}\label{aabb1}
  \bigl \| \lambda \bigl(A-\barint_{2^{k+1}Q} D\bigr) \nabla_x \mathcal{E}^*_{\lambda}  g\bigr\|_{\L^2( 2^{k+1} \Delta \setminus 2^{k} \Delta)}&\lesssim
  k^{(p-2)/(2p)}\|\lambda\nabla_x \mathcal{E}^*_{\lambda} g\|_{\L^2(2^{k+5} \Delta \setminus 2^{k-4} \Delta)}\notag\\
  &\lesssim
  k^{(p-2)/(2p)} e^{-\big(\frac{2^k \ell(\Delta)}{c\lambda}\big)} \lesssim  e^{-\big(\frac{2^k \ell(\Delta)}{c\lambda}\big)},
 \end{align}
 where we have also used off-diagonal estimates for $\lambda\nabla_x \mathcal{E}^*_{\lambda}$, that $\lambda\sim \ell(\Delta)$, and that $\|g\|_2 = 1$. Combining \eqref{aabb} and \eqref{aabb1},
 \begin{align}\label{aabb2}
 \|\mathcal{R}_\lambda A ((\nabla_x f)1_{2^{k+1} \Delta \setminus 2^k \Delta} )\|_{\L^2(\Delta)}&=\|\mathcal{U}_\lambda A ((\nabla_x f)1_{2^{k+1} \Delta \setminus 2^k \Delta} )\|_{\L^2(\Delta)}\notag\\
 &\lesssim  e^{-\big(\frac{2^k \ell(\Delta)}{c\lambda}\big)} \|\nabla_x  f\|_{\L^2({2^{k+1} \Delta \setminus 2^k \Delta})},
 \end{align}
 for $k\geq 11$. We can conclude that the proof of $\mathrm{(ii)}$ is complete.

  Having proved \eqref{keyedd} the rest of the argument is similar to proof of \cite[Lem. 2.27]{N1}. Indeed, let $ f \in  C_0^{\infty}(\ree)$, and let $\Delta$ be a parabolic cube satisfying $\lambda \sim \ell(\Delta)$. Note that
  \begin{align*}
  \mathcal{R}_{\lambda} \bigl(\nabla_x \bigl(f  - x \cdot \bariint_{\Delta} \nabla_x f \bigr) \bigr)=\mathcal{R}_{\lambda} (\nabla_x f)
  \end{align*}
  as $\mathcal{R}_{\lambda} 1=0$. Using this we have
 \begin{align*}
     \|\mathcal{R}_{\lambda} (\nabla_x f)\|_{2} & \lesssim \|\mathcal{R}_{\lambda} \bigl(\nabla_x \bigl(f  - x \cdot \bariint_{\Delta} \nabla_x f \bigr) 1_{\Delta} \bigr) \|_{2} \\ &+ \sum_{k=0}^{\infty}  \|\mathcal{R}_{\lambda} \bigl(\nabla_x \bigl(f- x \cdot \bariint_{ \Delta} \nabla_x f  \bigr) 1_{2^{k+1}\Delta \setminus 2^k \Delta}\bigr) \|_{2},
 \end{align*}
 by the triangle inequality, and hence
 \begin{align*}
     \|\mathcal{R}_{\lambda} (\nabla_x f)\|_{2}  & \lesssim \| \nabla_x \bigl( \bigl(f  - x \cdot \bariint_{\Delta} \nabla_x f \bigr) 1_{\Delta} \bigr) \|_{2} \\ &+   \sum_{k=0}^{\infty}  e^{-\big(\frac{2^k \ell(\Delta)}{c\lambda}\big)} \| \nabla_x \bigl( \bigl(f- x \cdot \bariint_{ \Delta} \nabla_x f  \bigr) 1_{2^{k+1}\Delta \setminus 2^k \Delta}\bigr) \|_{2},
 \end{align*}
 by \eqref{keyedd} $\mathrm{(i)}, \mathrm{(ii)}$. Continuing, using the Poincaré inequality, we deduce
 \begin{align*}
     \|\mathcal{R}_{\lambda} (\nabla_x f)\|_{2}  & \lesssim \|\lambda \nabla_x (\nabla_x f)\|_{2} + \|\lambda^2 \partial_t (\nabla_x f)\|_{2} \\&+ \sum_{k=0}^{\infty}  e^{-\big(\frac{2^k \ell(\Delta)}{c\lambda}\big)} \| \bigl(\nabla_x \bigl(  f- x \cdot \bariint_{ 2^{k+1} \Delta} \nabla_x f  \bigr) 1_{2^{k+1}\Delta \setminus 2^k \Delta}\bigr) \|_{2}   \\ &+ \sum_{k=0}^{\infty} \sum_{i=1}^k  e^{-\big(\frac{2^k \ell(\Delta)}{c\lambda}\big)} \| \nabla_x \bigl(x \cdot \bariint_{ 2^{i+1} \Delta} \nabla_x f- x \cdot \bariint_{ 2^{i} \Delta} \nabla_x f  \bigr) 1_{2^{k+1}\Delta \setminus 2^k \Delta} \|_{2}.
 \end{align*}
Furthermore,
 \begin{align*}
     & \sum_{k=0}^{\infty}  e^{-\big(\frac{2^k \ell(\Delta)}{c\lambda}\big)} \| \nabla_x \bigl(  f- x \cdot \bariint_{ 2^{k+1} \Delta} \nabla_x f  \bigr) 1_{2^{k+1}\Delta \setminus 2^k \Delta} \|_{2}  \\ & \lesssim \sum_{k=0}^{\infty}  e^{-\big(\frac{2^k \ell(\Delta)}{c\lambda}\big)} \| \nabla_x f -  \bariint_{ 2^{k+1} \Delta} \nabla_x f  \|_{\L^2(2^{k+1} \Delta)}  \\ & \lesssim    \|\lambda \nabla_x (\nabla_x f)\|_{2} + \|\lambda^2 \partial_t (\nabla_x f)\|_{2},
 \end{align*}
where we have used the Poincaré inequality in the last inequality. Note that
 \begin{align*}
    &  \sum_{k=0}^{\infty} \sum_{i=1}^k  e^{-\big(\frac{2^k \ell(\Delta)}{c\lambda}\big)} \| \nabla_x \bigl(x \cdot \bariint_{ 2^{i+1} \Delta} \nabla_x f- x \cdot \bariint_{ 2^{i} \Delta} \nabla_x f  \bigr) 1_{2^{k+1}\Delta \setminus 2^k \Delta} \|_{2} \\ & \lesssim \sum_{k=0}^{\infty} \sum_{i=1}^k  e^{-\big(\frac{2^k \ell(\Delta)}{c\lambda}\big)} \| \bariint_{ 2^{i+1} \Delta} \nabla_x f-  \bariint_{ 2^{i} \Delta} \nabla_x f   \|_{\L^2(2^{k+1} \Delta)}\\ & \lesssim \sum_{k=0}^{\infty} \sum_{i=1}^k  e^{-\big(\frac{2^k \ell(\Delta)}{c\lambda}\big)} 2^{k/2} |\Delta|^{1/2} \bigl| \bariint_{ 2^{i+1} \Delta}\bigl ( \nabla_x f-  \bariint_{ 2^{i} \Delta} \nabla_x f\bigr )\d x\d t \bigr |.
 \end{align*}
  Furthermore, continuing the estimate
 \begin{align*}
    &  \sum_{k=0}^{\infty} \sum_{i=1}^k  e^{-\big(\frac{2^k \ell(\Delta)}{c\lambda}\big)} 2^{k/2} |\Delta|^{1/2} \bigl| \bariint_{ 2^{i+1} \Delta}\bigl ( \nabla_x f-  \bariint_{ 2^{i} \Delta} \nabla_x f\bigr )\d x\d t \bigr | \\ & \lesssim \sum_{k=0}^{\infty} \sum_{i=1}^k  e^{-\big(\frac{2^k \ell(\Delta)}{c\lambda}\big)} 2^{k/2} |\Delta|^{1/2}  \bariint_{ 2^{i} \Delta} \bigl|\nabla_x f-  \bariint_{ 2^{i+1} \Delta} \nabla_x f \bigr|\d x\d t \\ & \lesssim \sum_{k=0}^{\infty} \sum_{i=1}^k  e^{-\big(\frac{2^k \ell(\Delta)}{c\lambda}\big)} 2^{k/2} |\Delta|^{1/2}  \bigl(\bariint_{ 2^{i+1} \Delta} \bigl|\nabla_x f-  \bariint_{ 2^{i+1} \Delta} \nabla_x f \bigr|^2\d x\d t\bigr)^{\frac{1}{2}}
    \\ & \lesssim  \|\lambda \nabla_x (\nabla_x f)\|_{2} + \|\lambda^2 \partial_t (\nabla_x f)\|_{2},
 \end{align*}
where we have used Cauchy-Schwarz inequality in the one to the last inequality, and the Poincaré inequality in the last inequality. Summing up the previous inequalities completes the proof of the lemma.
\end{proof}


\section{Proof of Theorem \ref{thm:Kato}}\label{sec2}

To prove Theorem \ref{thm:Kato} we first show that the proof can be reduce to the quadratic estimate
\begin{eqnarray}\label{kee}
      |||\lambda\cH \mathcal{E}_\lambda f|||_{2}\lesssim ||\mathbb Df||_{2}=\|\nabla_x f\|_{2}+ \| \dhalf f \|_{2},\   f\in \E(\ree).
      \end{eqnarray}
Recall that  $\mathcal{E}_\lambda = (1+ \lambda^2 \cH)^{-1}$ and that the norm $|||\cdot|||_{2}$ was introduced in \eqref{tnorm}.

In Lemma \ref{Lem: WP on R} we proved that the operator  $\cH$ is maximal accretive,  hence $\cH$ has  a bounded $H^\infty$-calculus, see \cite{Mc, Haase},  and $\cH$ has a unique maximal accretive square root $\sqrt{\cH}$ defined by the functional calculus for sectorial operators. The same is true for the adjoint $\cH^*$ and $\sqrt{\cH^*} = (\sqrt{\cH})^*$. Using \cite[Thm.~5.2.6]{Haase} we can express the unique maximal accretive square root $\sqrt{\cH}$ as
\begin{eqnarray}
\label{est1+kaintro}
\sqrt{\cH}f= \frac{16}{\pi} \int_0^\infty \lambda^3\cH^2 (1+\lambda^2\cH)^{-3} f\, \frac {\d\lambda}\lambda,
\end{eqnarray}
where $f \in \dom(\sqrt{\cH})$ and the integral is understood as an improper Riemann integral in $\L^2(\mathbb R^{n+1})$. We will use  this resolution formula for $\sqrt{\cH}$ to prove Theorem \ref{thm:Kato}. Indeed, testing this identity against $g \in \L^2(\mathbb R^{n+1})$ and applying Cauchy-Schwarz, we obtain
\begin{align*}
|\langle\sqrt{\cH}f,g\rangle_{2}|
&\leq \frac{16}{\pi}|||\lambda \cH(1+\lambda^2\cH)^{-1} f|||_{2}|||\lambda^2\cH^\ast
       (1+\lambda^2\cH^\ast)^{-2}g|||_{2}\\
       &=\frac{16}{\pi} |||\lambda\cH\mathcal{E}_\lambda  f|||_{2}|||\lambda^2\cH^\ast
       (1+\lambda^2\cH^\ast)^{-2}g|||_{2}.
\end{align*}
The second term is controlled by a structural constant times $\|g\|_{2,\mu}$ since $\cH^*$ is maximal accretive in $\L^2(\mathbb R^{n+1})$. Taking the supremum over all $g$ yields
\begin{align*}
    \|\sqrt{\cH}f\|_{2} \lesssim|||\lambda\cH \mathcal{E}_\lambda f|||_{2}.
\end{align*}
Assuming that \eqref{kee} holds we obtain
\begin{align}
\label{ga1}
 \| \sqrt{\cH}f \|_{2} \lesssim \|\nabla_x f\|_{2} +\|\HT \dhalf f \|_{2},
\end{align}
when $f$ is in  $\E(\mathbb R^{n+1}) \cap \dom(\sqrt{\cH}) \supset \dom(\cH)$. Since $\dom(\cH)$ is dense in $\E(\ree)$, see Lemma \ref{Lem: WP on R} $\mathrm{(iii)}$,  and as $\sqrt{\cH}$ is closed, the estimate extends to all $f \in \E(\mathbb R^{n+1}) $. As noted in \cite{AAEN},  $\cH^*$ is similar to an operator in the same class as $\cH$ under conjugation with the `time reversal' $f(t,x) \mapsto f(-t,x)$ and conjugation of $A$. Hence the above  reasoning applies to $\cH^*$ and hence we also have
\begin{align}
\label{ga2}
 \| \sqrt{\cH^*}f \|_{2} \lesssim \|\nabla_x f\|_{2} +\|\HT \dhalf f \|_{2},
\end{align}
whenever $f \in \E(\mathbb R^{n+1})$. In particular, assuming that \eqref{kee} and the corresponding estimate for $\cH^*$ holds, we can conclude that
\begin{align}\label{oneway}
 \|\sqrt {\cH}\, f\|_{2} + \|\sqrt {\cH^*}\, f\|_{2} \lesssim \|\nabla_x f\|_{2}+ \| \dhalf f \|_{2},
\end{align}
 for a every function $f \in \E(\ree)$.  Using \eqref{oneway}, and \eqref{eq2} with $\sigma = 0$ and $\delta$ small enough depending on the structural constants, we obtain for all $f \in \dom(\cH)$ that
 \begin{align*}
       \iint_{\ree} |\nabla_x f|^2  + |\dhalf f|^2\, \d x \d t & \lesssim_{\delta} | \langle \cH f\,, (1+\delta \HT) f \rangle| \\ & =  | \langle \sqrt{\cH} f\,, (\sqrt{\cH})^* (1+\delta \HT)f \rangle |\\ & \lesssim \| \sqrt{\cH} f\|_2  (\|\nabla_x f\|_{2}+ \| \dhalf f \|_{2}).
\end{align*}
 Hence,
\begin{align*}
     \|\nabla_x f\|_{2}+ \| \dhalf f \|_{2} \lesssim \| \sqrt{\cH} f\|_2,
\end{align*}
for all $f\in \dom(\cH)$. Since $\dom(\cH)$ is dense in $\dom(\sqrt{\cH})$ for the graph norm \cite[Prop.~3.1.1(h)]{Haase}, the estimate extends to all $f \in \dom(\sqrt{\cH})$. We can therefore conclude that the proof of Theorem \ref{thm:Kato} is reduced to the estimate in \eqref{kee}.

\subsection{Reduction of \eqref{kee} to a Carleson measure estimate}
      To start the proof of \eqref{kee}, we write
             \begin{eqnarray*}
     \lambda\cH\mathcal{E}_\lambda  f=\lambda\mathcal{E}_\lambda \cH f=\lambda\mathcal{E}_\lambda \cH\P_\lambda f+\lambda\mathcal{E}_\lambda \cH(I-\P_\lambda) f.
      \end{eqnarray*}
Now,
\begin{align*}
 \lambda^{-1}   (I+\lambda \cH) (I- \mathcal{E}_{\lambda}) f = \lambda^{-1}(I+\lambda \cH - I) f= \lambda \mathcal{H} f,
\end{align*}
for $f \in \dom(\cH)$. Hence,
\begin{align}
\label{eq:operatorequality}
  \lambda \mathcal{E}_{\lambda}  \cH  f= (I+\lambda \cH - I) f =  \lambda^{-1} \mathcal{E}_{\lambda} f,
\end{align}
for  $f \in \dom(\cH)$. Now, by Lemma \ref{Lem: WP on R} $\mathrm{(iii)}$ and Lemma \ref{le8-}  $\mathrm{(i)}$, the set $\dom(\cH)$ is dense in $\E(\R^{n+1})$ and the operator $\mathcal{E}_\lambda$ is $\L^2$-bounded. In conclusion, \eqref{eq:operatorequality} holds for every $f \in \E(\R^{n+1})$. Using the identity \eqref{eq:operatorequality},
      the $\L^2$-boundedness of $\mathcal{E}_\lambda$, and Lemma \ref{little2} $\mathrm{(ii)}$, we immediately see that
     \begin{eqnarray*}
    |||\lambda\mathcal{E}_\lambda \cH(I-\P_\lambda) f|||_{2}\lesssim |||\lambda^{-1}(I-\P_\lambda)f|||_{2}\lesssim\|\mathbb Df\|_{2}.
      \end{eqnarray*}
  Using the notation $\mathcal{U}_\lambda =\lambda\mathcal{E}_\lambda \div_x$,  we write
                   \begin{align*}
  \lambda\mathcal{E}_\lambda \cH\P_\lambda f&= -\mathcal{U}_\lambda A\nabla_x\P_\lambda f+\lambda\mathcal{E}_\lambda \dhalf \HT \dhalf \P_\lambda f.
      \end{align*}
      We note that
      \begin{eqnarray*}
  |||\lambda \mathcal{E}_\lambda\dhalf \HT \dhalf \P_\lambda f|||_2 \lesssim  |||\lambda(\dhalf P_\lambda) \HT \dhalf f|||_{2}\lesssim \|\HT \dhalf f\|_{2}\lesssim \|\mathbb Df\|_{2},
      \end{eqnarray*}
    where Lemma \ref{le8-} $\mathrm{(i)}$ is used in the first inequality and Lemma \ref{little1} is used in the second inequality. In the following,
    we will make use of the operator \begin{align*}
\mathcal{R}_\lambda &:=\mathcal{U}_\lambda A -(\mathcal{U}_\lambda A)\mathcal{A}_\lambda
 \end{align*}
 introduced in \eqref{opa1}. Using Lemma \ref{lemedda} we can conclude that $(\mathcal{U}_\lambda  A)$ is well-defined as an element in $\L^2_{\loc}(\mathbb R^{n+1})$. Therefore, using the operators $\mathcal{U}_\lambda$ and $\mathcal{R}_\lambda$, we write
      \begin{eqnarray*}
      \mathcal{U}_\lambda A\nabla_x\P_\lambda f=\mathcal{U}_\lambda A\P_\lambda\nabla_x f=\mathcal{R}_\lambda \P_{\lambda} \nabla_x f+(\mathcal{U}_\lambda A)\mathcal{A}_\lambda  \P_\lambda  \nabla_x f.\end{eqnarray*}
    Hence, using Lemma \ref{newes+}, we obtain that
      \begin{eqnarray*}
      |||\mathcal{R}_\lambda\P_\lambda \nabla_x f|||_{2}\lesssim |||\lambda \nabla_x (\P_\lambda \nabla_x f)|||_{2}+|||\lambda^{2} \partial_t(\P_\lambda \nabla_x f)|||_{2},
       \end{eqnarray*}
       uniformly in $\lambda$. Using this and Lemma \ref{little1}, we can conclude that
        \begin{eqnarray*}
      |||\mathcal{R}_\lambda \P_\lambda \nabla_x f|||_{2}\lesssim \|\mathbb D f\|_{2}.
       \end{eqnarray*}
       To conclude the proof of Theorem \ref{thm:Kato}, we use Lemma \ref{ilem2--}  stated and proved below. The lemma states that $$|\mathcal{U}_\lambda A|^2\, \frac{\d x\d t\d\lambda}{\lambda}$$ is a Carleson measure and that we have good control of the constants. Hence,
        \begin{eqnarray*}
      |||(\mathcal{U}_\lambda A)\mathcal{A}_\lambda \P_\lambda \nabla_x f|||_{2}\lesssim\|\mathbb D f\|_{2},
       \end{eqnarray*}
       by the dyadic Carleson's inequality, see \cite[Lem. 8.2]{AAEN}. This completes the proof of the estimate in \eqref{kee} and hence the proof of Theorem \ref{thm:Kato} modulo Lemma \ref{ilem2--}. The rest of the section is devoted to the proof of the remaining lemma, Lemma \ref{ilem2--}.

       \begin{lem} \label{ilem2--} Consider $A(x,t)=(A_1(x,t),\cdot \cdot \cdot,A_n(x,t))$ where $A_i \in \mathbb C^{n}$ for $i\in \{1,...,n\}.$ Define  $$\mathcal{U}_{\lambda} A := \Big(\mathcal{U}_{\lambda}  A_1,\cdot \cdot \cdot , \mathcal{U}_{\lambda} A_n \Big).$$
       Then,
\begin{eqnarray*}\label{crucacar+}\int_0^{\ell(\Delta)}\iint_{\Delta}|\mathcal{U}_\lambda A|^2\frac {\d x\d t \d\lambda}\lambda\lesssim |\Delta|,
\end{eqnarray*}
for all dyadic parabolic cubes $\Delta \subset\mathbb R^{n+1}$.
     \end{lem}

  The proof of Lemma \ref{ilem2--} is based on the use of appropriate local $Tb$-type test functions.

\subsection{Construction of appropriate local Tb-type test functions}\label{subtb} Let $\zeta\in \IC^{n}$ with $|\zeta|=1$ and let  $\zeta_{i}$ denote the $i$-th component of $\zeta$ for $1\leq i\leq n$.
We let $\chi, \eta$ be smooth functions on $\R^{n}$ and $\R$, respectively, whose values are in $[0,1]$. We construct the function $\chi$ to equal $1$ on $[-1/2,1/2]^n$ having support in $(-1,1)^{n}$, and the function $\eta$ to equal $1$ on $[-1/4,1/4]$ having support in $[-1,1]$. We fix a parabolic dyadic cube $\Delta$, and we denote its center by $(x_{\Delta}, t_{\Delta})$. We first introduce
\begin{align*}
 \chi_{\Delta}(x,t) &:= \chi \bigg( \frac{x-x_{\Delta}}{\ell(\Delta)}\bigg) \eta \bigg(\frac{t-t_{\Delta}}{\ell(\Delta)^2}\bigg).
 \end{align*}
Based on $\zeta$ and $\chi_{\Delta}$, we introduce
\begin{align*}
 L^\zeta_{\Delta}(x,t):= \chi_{\Delta}(x,t)(\Phi_\Delta(x)\cdot \overline{\zeta}),\ \Phi_\Delta(x):=(x-x_\Delta).
\end{align*}
Clearly, $L^\zeta_{\Delta} \in \E(\mathbb R^{n+1})$. Using the function $ L^\zeta_{\Delta}$, and  $0<\epsilon\ll 1$, we define the test function
\begin{align}\label{testfunction}
 f^\zeta_{\Delta,\epsilon}:=\mathcal{E}_{\epsilon\ell(\Delta)} L^\zeta_{\Delta}=(I+(\epsilon\ell(\Delta))^2\cH)^{-1} L^\zeta_{\Delta}.
 \end{align}
We need the following lemma for the T(b) type estimates.

\begin{lem}\label{laa} Let $\zeta\in \IC^{n}$ with $|\zeta|=1$, and let $0<\epsilon\ll 1$ be a degree of freedom. Given a parabolic dyadic cube $\Delta$, we assume that $f^\zeta_{\Delta,\epsilon }$ is the test function defined in \eqref{testfunction}. Then,
\begin{align*}
\mathrm{(i)}&\quad\iint_{\mathbb R^{n+1}}|f^\zeta_{\Delta, \epsilon }-L^\zeta_{\Delta}|^2\, \d x\d t \lesssim (\epsilon \ell(\Delta))^2|\Delta|,\notag\\
\mathrm{(ii)}&\quad\iint_{\mathbb R^{n+1}}|\mathbb D(f^\zeta_{\Delta,\epsilon}-L^\zeta_{\Delta})|^2\, \d x\d t \lesssim |\Delta|, \notag\\
\mathrm{(iii)}&\quad\|\nabla_x f^\zeta_{\Delta, \epsilon}\|^2_{2} + \|\dhalf f^\zeta_{\Delta, \epsilon}\|^2_{2} \lesssim |\Delta|.
\end{align*}
\end{lem}
\begin{proof} Note that
\begin{align*}
f^\zeta_{\Delta,\epsilon }-L^\zeta_{\Delta}&=-(\epsilon\ell(\Delta))^2\mathcal{E}_{\epsilon\ell(\Delta)}\cH L^\zeta_{\Delta}\notag\\
&=
-(\epsilon\ell(\Delta))^2\mathcal{E}_{\epsilon\ell(\Delta)}\dhalf H_t \dhalf L^\zeta_{\Delta}+(\epsilon\ell(\Delta))^2\mathcal{E}_{\epsilon\ell(\Delta)}
\div_x(A\nabla_x L^\zeta_{\Delta}),\end{align*}
where we used the identity $\partial_t = \dhalf H_t \dhalf.$ Consider $g \in \L^2(\mathbb{R}^{n+1})$ and note that
\begin{align*}
    &\iint_{\mathbb{R}^{n+1}}
\mathcal{E}_{\epsilon\ell(\Delta)}
\div_x(A\nabla_x L^\zeta_{\Delta})(x,t)\, \overline{g}(x,t) \, \d x \d t  =-\iint_{\mathbb{R}^{n+1}}
A\nabla_x L^\zeta_{\Delta}(x,t)\, \overline{\nabla_x \mathcal {E}^*_{\epsilon\ell(\Delta)}g}(x,t) \, \d x \d t.
\end{align*}
Consequently, using Lemma $\ref{p3.1}$ in conjunction with \eqref{eqbmo}, and Lemma \ref{le8-}, we deduce that
\begin{align*}
    &\biggl |\iint_{\mathbb{R}^{n+1}}
\mathcal{E}_{\epsilon\ell(\Delta)}
\div_x(A\nabla_x L^\zeta_{\Delta})(x,t)\, \overline{g}(x,t) \, \d x \d t \biggr |&{\lesssim}
\|\nabla_x L^\zeta_{\Delta}\|_{2} \|\nabla_x \mathcal {E}^*_{\epsilon\ell(\Delta)}g\|_{2} \lesssim \frac{1}{\epsilon \ell(\Delta)} \|\nabla_x L^{\zeta}_{\Delta} \|_{2} \|g\|_{2},
\end{align*} for every $g \in \L^2(\mathbb{R}^{n+1})$. Hence, \begin{align*}
    \|\mathcal{E}_{\epsilon\ell(\Delta)}
\div_x(A\nabla_x L^\zeta_{\Delta})\|_{2} {\lesssim} \frac{1}{\epsilon \ell(\Delta)} \|\nabla_x L^{\zeta}_{\Delta} \|_{2}.
\end{align*}
Putting estimates together, using also the uniform $\L^2$-boundedness of the operator $(\epsilon\ell(\Delta))\mathcal{E}_{\epsilon\ell(\Delta)}\dhalf$, see Lemma \ref{le8-}, we obtain
\begin{align*}
\iint_{\mathbb R^{n+1}}|f^\zeta_{\Delta,\epsilon }-L^\zeta_{\Delta}|^2\, \d x\d t &\lesssim \iint_{\mathbb R^{n+1}}|(\epsilon\ell(\Delta))\mathbb DL^\zeta_{\Delta}|^2\, \d x\d t.
\end{align*}
Furthermore,
\begin{align*}
\iint_{\mathbb R^{n+1}}|\mathbb DL^\zeta_{\Delta}|^2\, \d x\d t &\lesssim \iint_{\mathbb R^{n+1}}|\nabla_x L^\zeta_{\Delta}|^2\, \d x\d t +\iint_{\mathbb R^{n+1}}|\HT\dhalf L^\zeta_{\Delta}|^2\, \d x\d t \lesssim |\Delta|,
\end{align*}
by the construction of $L^\zeta_{\Delta}$, and elementary estimates for $\HT\dhalf L^\zeta_{\Delta}$ based on homogeneity. Similarly, we deduce that
\begin{eqnarray*}
\iint_{\mathbb R^{n+1}}|\mathbb D(f^\zeta_{\Delta,\epsilon}-L^\zeta_{\Delta})|^2\, \d x\d t \lesssim |\Delta|.
\end{eqnarray*}
This proves $\mathrm{(i)}$ and $\mathrm{(ii)}$. To prove $\mathrm{(iii)}$, we simply use $\mathrm{(ii)}$ and note that
\begin{align*}
\iint_{\mathbb R^{n+1}}|\mathbb Df^\zeta_{\Delta,\epsilon }|^2\, \d x\d t &\lesssim |\Delta|+\iint_{\mathbb R^{n+1}}|\mathbb D L^\zeta_{\Delta}|^2\, \d x\d t \lesssim |\Delta|.
\end{align*}
This proves $\mathrm{(iii)}$.
\end{proof}

\subsection{Verifying the Carleson measure estimate} The proof of Lemma \ref{ilem2--}, and hence the remainder of the proof follows, exactly as in \cite{AAEN}. Hence, we will only outline the crucial components of the proof.
\begin{lem} \label{ilem2--+} Given a parabolic dyadic cube $\Delta$, let $f^\zeta_{\Delta,\epsilon }$ be the test function defined in  \eqref{testfunction}. There exists $\epsilon\in (0,1)$, depending only on the structural constants, and a finite set $W$ of unit vectors in
     $\mathbb C^{n}$, whose
     cardinality depends on $\epsilon$ and $n$, such that
     \begin{eqnarray*}
 \sup_{\Delta} \frac 1{|\Delta|}\int_0^{\ell(\Delta)}\iint_{\Delta}|\mathcal{U}_\lambda A|^2\frac {\d x\d t \d\lambda} \lambda\lesssim\sum_{\zeta\in W} \sup_{\Delta} \frac 1{|\Delta|}\int_0^{\ell(\Delta)}\iint_{\Delta}|(\mathcal{U}_\lambda A)\mathcal{A}_\lambda \nabla_x f^\zeta_{\Delta,\epsilon }|^2\frac {\d x\d t \d\lambda}\lambda,
     \end{eqnarray*}
 where the supremum is taken over all dyadic parabolic cubes $\Delta \subset \ree$ and $\mathcal{A}_\lambda $ is the dyadic averaging operator induced by $\Delta$ and defined in \eqref{dy}.
     \end{lem}
\begin{proof} For the proof see \cite[Lem. 8.5]{AAEN}. In fact, the assumption that $A$ has $\BMO $ coefficients does not add complication to this part of  the argument, as the argument relies on Lemma \ref{lemedda} and Lemma \ref{laa} only. We here therefore omit further details.
\end{proof}

To continue the proof of Lemma \ref{ilem2--}, we first note that Lemma \ref{ilem2--+} implies
     that it suffices to prove that
\begin{eqnarray}\label{ff1}
\int_0^{\ell(\Delta)}\iint_\Delta|(\mathcal{U}_\lambda A)\mathcal{A}_\lambda \nabla_x f^\zeta_{\Delta,\epsilon }|^2\frac {\d x\d t \d\lambda}\lambda\lesssim |\Delta|,
\end{eqnarray}
for all dyadic parabolic cubes $\Delta\subset\mathbb R^{n+1}$ and for each $\zeta\in W$. In the following, we will simply, with a slight abuse of notation but consistently, drop the $\cdot$ in \eqref{ff1}. We write
\begin{align*}
(\mathcal{U}_\lambda A)\mathcal{A}_\lambda \nabla_x f^\zeta_{\Delta,\epsilon }=(\mathcal{U}_\lambda A\nabla_x f^\zeta_{\Delta,\epsilon }-\lambda\mathcal{E}_\lambda\dhalf H_t \dhalf f^\zeta_{\Delta,\epsilon })+(-\mathcal{R}_\lambda \nabla_x f^\zeta_{\Delta,\epsilon }+\lambda\mathcal{E}_\lambda \dhalf H_t \dhalf f^\zeta_{\Delta, \epsilon }),
\end{align*}
where
\begin{align*}
      \mathcal{R}_\lambda=\mathcal{U}_\lambda A -(\mathcal{U}_\lambda A) \mathcal{A}_\lambda.
\end{align*}
To proceed, we first note that
\begin{eqnarray*}
\mathcal{U}_\lambda A\nabla_x f^\zeta_{\Delta, \epsilon }-\lambda\mathcal{E}_\lambda\dhalf H_t \dhalf f^\zeta_{\Delta, \epsilon }=-\lambda \mathcal{E}_\lambda\mathcal{H}f^\zeta_{\Delta, \epsilon }.
\end{eqnarray*}
Hence,  using the $\L^2$-boundedness of $\mathcal{E}_\lambda$, see Lemma \ref{le8-} $\mathrm{(i)}$,
\begin{align*}
\int_0^{\ell(\Delta)}\iint_{\Delta}|\mathcal{U}_\lambda A\nabla_x f^\zeta_{\Delta,\epsilon }-\lambda\mathcal{E}_\lambda\dhalf H_t \dhalf f^\zeta_{\Delta,\epsilon }|^2\frac {\d x\d t\d\lambda}\lambda&\lesssim \ell(\Delta)^2
\iint_{\Delta}|\cH f^\zeta_{\Delta,\epsilon }|^2{\d x\d t},
\end{align*}
as we are only integrating in $\lambda$ over the range $[0,\ell(\Delta)]$. However,
$$\cH f^\zeta_{\Delta,\epsilon }=( L^\zeta_{\Delta}-f^\zeta_{\Delta,\epsilon })/(\epsilon \ell(\Delta))^2,$$
and hence
\begin{align*}
\int_0^{\ell(\Delta)}\iint_{\Delta}|\mathcal{U}_\lambda A\nabla_x f^\zeta_{\Delta,\epsilon }-\lambda\mathcal{E}_\lambda\dhalf H_t \dhalf f^\zeta_{\Delta,\epsilon }|^2\frac {\d x\d t\d\lambda}\lambda&\lesssim \epsilon^{-4}\ell(\Delta)^{-2}
\iint_{\mathbb R^{n+1}}|f^\zeta_{\Delta,\epsilon }-L^\zeta_{\Delta}|^2{\d x\d t}\\
&\lesssim \epsilon^{-2}|\Delta|,
\end{align*}
by Lemma \ref{laa} $\mathrm{(i)}$. Second, using that $\mathcal{A}_\lambda \mathcal{A}_\lambda =\mathcal{A}_\lambda $, we write

\begin{align*}
-\mathcal{R}_\lambda \nabla_x f^\zeta_{\Delta,\epsilon }+\lambda\mathcal{E}_\lambda\dhalf H_t \dhalf f^\zeta_{\Delta, \epsilon } &= -\mathcal{R}_\lambda \P_{\lambda} \nabla_x f^\zeta_{\Delta,\epsilon } \\ &-\mathcal{R}_\lambda (1-\P_{\lambda}) \nabla_x f^\zeta_{\Delta,\epsilon } +\lambda\mathcal{E}_\lambda\dhalf H_t \dhalf f^\zeta_{\Delta, \epsilon }\\
&= -\mathcal{R}_\lambda \P_{\lambda} \nabla_x f^\zeta_{\Delta,\epsilon } +
(\mathcal{U}_\lambda A)\mathcal{A}_\lambda (\mathcal{A}_\lambda - \P_\lambda) \nabla_x f^\zeta_{\Delta,\epsilon }  \\&- \mathcal{U}_\lambda A (1-\P_{\lambda}) \nabla_x f^\zeta_{\Delta, \epsilon }+\lambda\mathcal{E}_\lambda\dhalf H_t \dhalf f^\zeta_{\Delta, \epsilon }.
\end{align*}
Using Lemma \ref{newes+},  we obtain
\begin{align*}
    \int _{0}^{\ell(\Delta)}\iint_{\Delta} |\mathcal{R}_\lambda \P_{\lambda} \nabla_x f^\zeta_{\Delta,\epsilon }|^2 \, \frac{\d x \d t \d \lambda}{\lambda} & =    \int _{0}^{\ell(\Delta)}\iint_{\Delta} |\mathcal{R}_\lambda  \nabla_x \P_{\lambda} f^\zeta_{\Delta,\epsilon }|^2 \, \frac{\d x \d t \d \lambda}{\lambda} \\ &\lesssim \int_{0}^{\ell(\Delta)}
\iint_{\mathbb R^{n+1}}| \nabla_x \P_\lambda\nabla_x f^\zeta_{\Delta,\epsilon }|^2\, \lambda{\d x\d t\d\lambda}\notag\\
&+\int_{0}^{\ell(\Delta)}\iint_{\mathbb R^{n+1}}| \partial_t \P_\lambda\nabla_x f^\zeta_{\Delta,\epsilon }|^2\, \lambda^{3}{\d x\d t\d\lambda}.
\end{align*}
 In particular, by Lemma \ref{little1}, we can conclude that
\begin{eqnarray*}
\int_{0}^{\ell(\Delta)}\iint_{\Delta}|\mathcal{R}_\lambda \P_\lambda\nabla_xf^\zeta_{\Delta, \epsilon }|^2\, \frac {\d x\d t\d\lambda}\lambda\lesssim \iint_{\mathbb R^{n+1}}|\nabla_x f^\zeta_{\Delta, \epsilon }|^2\, \d x\d t\lesssim |\Delta|,
\end{eqnarray*}
where we used Lemma \ref{laa} $\mathrm{(iii)}$. Using Lemma \ref{le11+}, we see that
$$\|(\mathcal{U}_\lambda  A)\mathcal{A}_\lambda \|_2\lesssim 1.$$
Thus,
\begin{align*}
\int _{0}^{\ell(\Delta)}\iint_{\Delta}| (\mathcal{U}_\lambda  A)\mathcal{A}_\lambda (\mathcal{A}_\lambda -\P_\lambda)\nabla_x f^\zeta_{\Delta,\epsilon }|^2\, \frac {\d x\d t\d\lambda}\lambda&\lesssim \int_{\mathbb R}\iint_{\mathbb R^{n+1}}|(\mathcal{A}_\lambda -\P_\lambda)\nabla_xf^\zeta_{\Delta,\epsilon }|^2\, \frac {\d x\d t\d\lambda}\lambda\notag\\
&\lesssim \iint_{\mathbb R^{n+1}}|\nabla_x f^\zeta_{\Delta,\epsilon }|^2\, {\d x\d t}\lesssim |\Delta|,
\end{align*}
where we have used Lemma \ref{little3} in the second inequality and  Lemma \ref{laa} $\mathrm{(iii)}$ in the last inequality. Furthermore,
\begin{align*}
-\mathcal{U}_\lambda A (I-\P_\lambda) \nabla_x f^\zeta_{\Delta,\epsilon }+\lambda\mathcal{E}_\lambda\dhalf H_t \dhalf f^\zeta_{\Delta,\epsilon }&=\lambda\mathcal{E}_\lambda\cH (I-\P_\lambda) f^\zeta_{\Delta,\epsilon }+\lambda\mathcal{E}_\lambda \dhalf H_t \dhalf \P_\lambda f^\zeta_{\Delta,\epsilon }.
\end{align*}
By the $\L^2$-boundedness of $\mathcal{E}_\lambda$ in Lemma $\ref{le8-}$ $\mathrm{(i)}$ and \eqref{eq:operatorequality}, we see that
\begin{align*}
\int_{0}^{\ell(\Delta)}\iint_{\Delta}|\lambda\mathcal{E}_\lambda\cH (I-\P_\lambda) f^\zeta_{\Delta,\epsilon }|^2\, \frac {\d x\d t\d\lambda}\lambda&\lesssim
\int_{\mathbb R}\iint_{\mathbb R^{n+1}}| \lambda^{-1}(I-\P_\lambda)f^\zeta_{\Delta,\epsilon }|^2\, \frac {\d x\d t\d\lambda}\lambda\\
&\lesssim \iint_{\mathbb R^{n+1}}|{\mathbb D} f^\zeta_{\Delta,\epsilon }|^2\, \d x\d t\lesssim |\Delta|,
\end{align*}
by Lemma \ref{little2} $\mathrm{(ii)}$ and Lemma \ref{laa} $\mathrm{(iii)}$. We also deduce
\begin{align*}
\int_{0}^{\ell(\Delta)}\iint_{\Delta}|\lambda\mathcal{E}_\lambda\dhalf H_t \dhalf \P_\lambda f^\zeta_{\Delta,\epsilon }|^2\, \frac {\d x\d t\d\lambda}\lambda&\lesssim
\int_{\mathbb R}\iint_{\mathbb R^{n+1}}| \lambda \dhalf H_t \dhalf \P_\lambda f^\zeta_{\Delta,\epsilon }|^2\, \frac {\d x\d t\d\lambda}\lambda\\
&\lesssim \int_{\mathbb R}\iint_{\mathbb R^{n+1}}|\lambda (\dhalf \mathcal{P}_\lambda) (\HT \dhalf f^\zeta_{\Delta,\epsilon })|^2\, \frac {\d x\d t\d\lambda}\lambda\\
&\lesssim \|\HT \dhalf f^\zeta_{\Delta,\epsilon }\|_{2}\lesssim |\Delta|,
\end{align*}
where we used the $\L^2$ boundedness of $\mathcal{E}_\lambda$,  Lemma $\ref{little1}$, the Plancherel's theorem, and Lemma $\ref{laa}$ $\mathrm{(iii)}$. Together these estimates prove \eqref{ff1}, hence completing the proof of  Lemma \ref{ilem2--}.

\bigskip
\noindent
{\bf Acknowledgment.} We sincerely thank the anonymous referees for their thorough review of the manuscript and their insightful comments. Their valuable feedback has significantly contributed to improving the quality and clarity of the paper.

\def\cprime{$'$} \def\cprime{$'$} \def\cprime{$'$}

\end{document}